\documentclass[11pt,a4paper,notitlepage,
psamsfonts]{amsart}

\usepackage{amssymb,amsfonts,amsmath}\usepackage[all,arc]{xy}
\usepackage{enumerate}
\usepackage{mathrsfs}
\usepackage{fullpage}
\usepackage{xspace}
\usepackage[margin=1.0in]{geometry}
\usepackage{tcolorbox}
\usepackage{tikz-cd}
\usepackage{color}
\usepackage{aliascnt}
\usepackage{hyperref}

\newtheorem{thm}{Theorem}[section]

\newaliascnt{cor}{thm}
\newtheorem{cor}[cor]{Corollary}
\aliascntresetthe{cor}

\newaliascnt{prop}{thm}
\newtheorem{prop}[prop]{Proposition}
\aliascntresetthe{prop}

\newaliascnt{lem}{thm}
\newtheorem{lem}[lem]{Lemma}
\aliascntresetthe{lem}

\newaliascnt{conj}{thm}

\aliascntresetthe{conj}

\newaliascnt{que}{thm}
\newtheorem{que}[que]{Question}
\aliascntresetthe{que}

\theoremstyle{definition}
\newtheorem{defn}[thm]{Definition}

\newtheorem{exmp}[thm]{Example}

\newtheorem{notn}[thm]{Notation}

\newtheorem{conv}[thm]{Convention}
\newtheorem{ass}[thm]{Assumption}

\newtheorem{expectation}[thm]{Expectation}

\theoremstyle{remark}
\newtheorem{rem}[thm]{Remark}

\newcommand{\Z}{\mathbb{Z}\xspace}
\newcommand{\Q}{\mathbb{Q}\xspace}
\newcommand{\G}{\mathbb{G}\xspace}
\DeclareMathOperator{\Spec}{Spec}

\DeclareMathOperator{\alb}{alb}
\DeclareMathOperator{\img}{Im}

\DeclareMathOperator{\ck}{coker}

\DeclareMathOperator{\Cor}{Cor}

\DeclareMathOperator{\Hom}{Hom}
\DeclareMathOperator{\Pic}{Pic}
\DeclareMathOperator{\Gal}{Gal}
\DeclareMathOperator{\fc}{frac}
\DeclareMathOperator{\Fil}{Fil}
\DeclareMathOperator{\Mod}{Mod}
\DeclareMathOperator{\Cone}{Cone}
\DeclareMathOperator{\FI}{FI}
\DeclareMathOperator{\Gp}{Gp}
\makeatletter
\let\c@equation\c@thm
\makeatother
\numberwithin{equation}{section}

\title{A finer Tate duality theorem for local Galois symbols}
\author{Evangelia Gazaki}
\email{gazaki@umich.edu}
\address{Department of Mathematics\\ University of Michigan\\ 530 Church Street, Ann Arbor, Michigan 48109}
\begin{document}

\maketitle
\begin{abstract}

Let $K$ be a finite extension of $\Q_p$. Let $A$, $B$ be abelian varieties over $K$ with good reduction. For any integer $m\geq 1$, we consider the  Galois symbol $K(K;A,B)/m\rightarrow H^2(K,A[m]\otimes B[m])$, where $K(K;A,B)$ is the Somekawa $K$-group attached to $A,B$. This map is a generalization of the Galois symbol $K_2^M(K)/m\rightarrow H^2(K,\mu_m^{\otimes 2})$ of the Bloch-Kato conjecture, where $K_2^M(K)$ is the Milnor $K$-group of $K$. In this paper we give a geometric description of the image of this generalized Galois symbol by looking at the Tate duality pairing $H^{2}(K,A[m]\otimes B[m])\times\Hom_{G_{K}}(A[m],B^{\star}[m])\rightarrow\Z/m,$ where $B^\star$ is the dual abelian variety of $B$. Under this perfect pairing we compute the exact annihilator of the image of the Galois symbol in terms of an object of integral $p$-adic Hodge theory. In this way we generalize a result of Tate for $H^1$. Moreover, our result has applications to zero cycles on abelian varieties.
\end{abstract}

\section{Introduction} The main goal of this article is to generalize a classical result of J. Tate about local Tate duality. We start by reviewing this classical result.

\subsection{Finer Tate duality for $H^1$} Let $K$ be a finite extension of the $p$-adic field $\Q_{p}$ with Galois group $G_{K}:=\Gal(\overline{K}/K)$. Local Tate duality tells us that for a finite continuous $G_{K}$-module $M$ annihilated by a positive integer $m$, there is a canonical perfect pairing of finite abelian groups
\[H^{i}(K,M)\times H^{2-i}(K,M^{\star})\longrightarrow\Z/m\Z,\] where by $H^{i}(K,-)$ we denote  the Galois cohomology of $K$, and $M^{\star}=\Hom(M,\mu_{m})$.

In the special case when $M=A[m]$ is the Galois module of the $m$-torsion points of an abelian  variety $A$ over $K$, Tate has a finer result, which we recall here.
The short exact sequence of $G_K$-modules, $0\longrightarrow A[m]\longrightarrow A(\overline{K})\stackrel{m}{\longrightarrow}A(\overline{K})\rightarrow 0$, known as the Kummer sequence of $A$, induces a connecting homomorphism $A(K)\stackrel{\delta}{\longrightarrow} H^{1}(K,A[m])$. The Weil pairing gives an isomorphism $(A[m])^{\star}\simeq A^{\star}[m]$, where $A^{\star}$ is the dual abelian variety of $A$. The local Tate duality pairing for $H^{1}$ therefore takes on the form
$H^{1}(K,A[m])\times H^{1}(K,A^{\star}[m])\longrightarrow\Z/m.$
Tate proved that under this perfect pairing the orthogonal complement of $\img(A(K)\stackrel{\delta}{\rightarrow} H^{1}(K,A[m]))$ is the corresponding part, $\img(A^{\star}(K)\stackrel{\delta}{\rightarrow} H^{1}(K,A^{\star}[m]))$, that comes from the points.

We will refer to this classical result as the finer local Tate duality for $H^{1}$. The main goal of this paper is to describe a finer local Tate duality for $H^{2}$. To do this we need a right set-up, that is a suitable finite $G_K$-module $M$ and a significant subgroup $H$ of $H^2(K,M)$. \vspace{2pt}
\subsection{Set up and Statement of the main theorem}\label{intro1}
We claim that a suitable set-up for the $H^2$ problem is the following. Let $A$, $B$ be abelian varieties over the $p$-adic field $K$ and $m\geq 1$ a positive integer. We consider the finite $G_K$-module $M=A[m]\otimes B[m]$. For any finite extension $L$ over $K$,  the Kummer sequences on $A, B$ and  the cup product induce a map
\[A(L)\otimes B(L)\longrightarrow H^{2}(L,A[m]\otimes B[m])\longrightarrow
H^{2}(K,A[m]\otimes B[m]),\] where the second map is the Corestriction map of Galois cohomology. Hence, we obtain a map 
\[\bigoplus_{L/K\text{ finite}}A(L)\otimes B(L)\longrightarrow H^{2}(K,A[m]\otimes B[m]).\] The image of this map is the \textit{significant subgroup} of $H^2(K,M)$, whose orthogonal complement under Tate duality we wish to describe. 

To see why this is a  good set-up, we note  that the above map factors through a certain $K$-group, namely the Somekawa $K$-group $K(K;A,B)$ attached to the abelian varieties $A,B$ (see \cite{Som}). This group is a generalization of the Milnor K-group, $K_2^M(K)$, and the map \[K(K;A,B)/m\xrightarrow{s_m} H^{2}(K,A[m]\otimes B[m])\] is a generalization of the classical Galois symbol
$K_2^M(K)/m\rightarrow H^2(K,\mu_m^{\otimes 2})$ of the Milnor conjecture. In fact, we call the map $s_m$ the \textit{generalized Galois symbol}. We note that  for a single abelian variety $A$, the group $K(K;A)$ is isomorphic to $A(K)$, which together with the $G_K$-module $A[m]$, form the  set-up for the $H^1$-problem considered by Tate. 

The image of this Galois symbol $s_m$ is the part of $H^2(K,A[m]\otimes B[m])$ that comes from the points of the abelian varieties and  this is the part we want to understand. We consider the Tate duality perfect pairing, 
\[(\star):\;H^{2}(K,A[m]\otimes B[m])\times\Hom_{G_{K}}(A[m],B^{\star}[m])\longrightarrow\Z/m,\] where $B^{\star}$ is the dual abelian variety of $B$. The main theorem of this article is the following.
\begin{thm}\label{BIG}(Finer Tate duality for $H^{2}$) Let $K$ be a finite extension of $\Q_{p}$, where $p>3$. Let $A$, $B$ be abelian varieties over $K$ with good reduction. Let $\mathcal{A}$, $\mathcal{B}^{\star}$ be the N\'{e}ron models of $A$ and $B^{\star}$ respectively. Under the Tate duality pairing, the orthogonal complement of $\img(K(K;A,B)\rightarrow H^{2}(K,A[m]\otimes B[m]))$ consists of those homomorphisms $f$ that lift to a  homomorphism $\mathcal{A}[m]\xrightarrow{\tilde{f}}\mathcal{B}^{\star}[m]$ of finite flat group schemes over the ring of integers, $\mathcal{O}_{K}$, of $K$.
\end{thm}

When $m$ is coprime to $p$ the proof follows easily by a computation of Raskind and Spiess (\cite{Ras/Sp}) and it is unconditional on the prime $p$. When $m=p^n$ though the problem becomes much harder and this case will occupy most of this article. Our method involves computations with integral crystalline cohomology groups in the sense of Faltings (\cite{Fal}). Unfortunately this imposes the limitation $p>3$. We believe that the statement of theorem \ref{BIG} should be true for every prime $p$. In fact, in the last section we discuss some specific examples where one can give an alternative proof of theorem \ref{BIG} that holds without any limitation on $p$.
\vspace{2pt}
\subsection{History of local Galois symbols} Let us briefly recall the history behind the study of Galois symbols. 

\subsubsection{Milnor $K$-theory and the motivic Bloch-Kato conjecture} For a general field $K$ (not nece-ssarily $p$-adic) the higher Milnor $K$-groups of $K$ are defined to be the graded pieces of 
\[K^M_\star(K):=\frac{T^\star K^\times}{<a\otimes(1-a)>},\] where $T^\star K^\times$ is the tensor algebra over the multiplicative group $K^\times$ and $<a\otimes(1-a)>$ is the two-sided ideal generated by elements of the form $a\otimes(1-a)$, for $a\in K$ with $a\neq 0,1$.  The elements of $K^M_r(K)$ are traditionally denoted as symbols, that is linear combinations of $\{x_1,\cdots,x_r\}$, with $x_i\in K^\times$ for $i\in\{1,\cdots,r\}$. For $r\geq 0$ the group $K_r^M(K)$   is a geometric invariant of $K$, which has also a realization as motivic cohomology.

Moreover, there is a map to Galois cohomology, known as the Galois symbol of the motivic Bloch-Kato conjecture,
\[K^{M}_{r}(K)/m\rightarrow H^{r}(K,\mu_{m}^{\otimes r}).\] The latter is always an isomorphism (when $r=2$ this is due to Merkurjev and Suslin (\cite{Mer/Sus} and the general case is due to Voevodsky (\cite{Voevodsky2011})). The above, nowadays known as the  \textit{norm-residue isomorphism theorem},  has tremendous applications to both arithmetic and geometric problems such as class field theory and the Birch and Swinnerton-Dyer conjecture.
\vspace{1pt}
\subsubsection{Class Field Theory} 

S. Bloch in  \cite{Bloch2}  introduced a variant of $K_2^M(K)$, in order to study class field theory of arithmetic surfaces. We denote this variant  by $K(K;J,\G_m)$, where $J$ is the Jacobian variety of a smooth complete curve over a  field $K$. The key for his main theorem was to show  surjectivity of the corresponding local Galois symbol, $K(K;J,\G_m)/p^n\rightarrow H^2(K,J[p^n]\otimes\mu_{p^n})$, where $K$ is $p$-adic and the Jacobian  $J$ is assumed to have good reduction.
\vspace{1pt}
\subsubsection{Somekawa K-groups}  Later Somekawa (\cite{Som}), following a suggestion of K. Kato gene-ralized the above definition allowing more general coordinates. Namely, he defined a $K$-group $K(K;G_1,\cdots,G_r)$  attached to semi-abelian varieties $G_1,\cdots,G_r$ over a field $K$ (that is extensions of abelian varieties by torii). The group $K(K;G_1,G_2,\cdots, G_r)$ is a quotient of $\displaystyle\bigoplus_{L\supset K\text{ finite}}G_1(L)\otimes\cdots\otimes G_r(L)$  by two geometric relations and it is expected to have a motivic interpretation.\footnote{In fact it has been shown (\cite{KY}) that $K(K;G_1,\cdots,G_r)\stackrel{\simeq}{\longrightarrow} 
\Hom_{DM^{eff}}(\Z,G_1[0]\otimes\cdots\otimes G_r[0])$.}. In the special case when  $G_i=\G_m$ for all $i$, we recover the Milnor $K$-group  $K_{r}^{M}(K)$.  Moreover, for an integer $m\geq 1$ invertible in $K$, there is a generalized Galois symbol,
\begin{eqnarray}\label{symbol} K(K;G_1,\cdots,G_r)/m\stackrel{s_m}{\longrightarrow} H^r(K,G_1[m]\otimes\cdots\otimes G_r[m]).\end{eqnarray}

\subsubsection{Philosophy}
The Somekawa K-groups are objects of geometric interest that have been related to algebraic cycles. Via the Galois symbol, $K(K;G_1,\cdots, G_r)$ relates to a better understood algebraic object. In particular, the maps $s_m$ can be realized as special cases of cycle maps to \'{e}tale cohomology. In the philosophy of the Tate conjecture, it is an interesting question to see which classes in Galois cohomology arise from the image of the Galois symbol.

In particular for a $p$-adic field $K$ and an integer $r\geq 3$, the above map vanishes. When $r=2$ though, the
local Galois symbols $K(K;G_1,G_2)/m\stackrel{s_m}{\longrightarrow} H^2(K,G_1[m]\otimes G_2[m])$, for various choices of semi-abelian $G_1,G_2$, carry a lot of information and have been studied by numerous people in different frameworks.


\subsubsection{Zero cycles} The reason we are particularly  interested in the group $K(K;A,B)$ attached to abelian varieties $A,B$ is because it has significant applications to zero cycles on products of curves and abelian varieties. 

For a smooth projective variety $X$ over $K$, the group $CH_0(X)$ of zero cycles modulo rational equivalence has a filtration $CH_0(X)\supset A_0(X)\supset T(X)$. When $X$ has a $K$-rational point, the first quotient of the filtration is isomorphic to $\Z$, while the second is isomorphic to the $K$-rational points of an abelian variety. The group $T(X)$ is called the Albanese kernel of $X$ and it is the most wild part of the Chow group, over which we do not have enough control. The group $K(K;A,B)$ has been related to the Albanese kernel on products of curves by Raskind and Spiess (\cite{Ras/Sp}) and on abelian varieties by the current author (\cite{Gaz12015}). The simplest example is that of a product of two elliptic curves $E_{1}$, $E_{2}$ over $K$. In this case the group $K(K;E_{1},E_{2})$ turns out to be isomorphic to the Albanese kernel, $T(E_{1}\times E_{2})$, and the Galois symbol $s_{m}$ coincides with the cycle map to \'{e}tale cohomology $CH_{0}(E_{1}\times E_{2})/m\longrightarrow H^{4}((
E_{1}\times E_{2})_{et},\mu_{m}^{\otimes 2})$, when the latter is restricted to $T(E_{1}\times E_{2})$.

The map $K(K;A,B)\stackrel{s_{p^n}}{\longrightarrow}H^2(K,A[p^n]\otimes B[p^n])$ has been previously studied by numerous people, some examples including \cite{Ras/Sp}, \cite{Yam}, \cite{Mur/Ram}, \cite{Hir3}, \cite{Hir1}. The results though are only partial, imposing in most cases the strong $K$-rationality assumption that $A[p^n]\subset A(K)$ and $B[p^n]\subset B(K)$. Moreove, they  mostly cover the case of two elliptic curves, $E_1,E_2$. All these results required very deep computations of symbols and what  they all suggested is that the image of $s_{p^n}$ tends to be small relative to $H^2(K,A[p^n]\otimes B[p^n])$ and it depends very heavily on the ramification of $K$ and the reduction type of the abelian varieties. 



\subsection{Advantages of our approach and Applications}
Our main goals with theorem \ref{BIG} are first to unify all the different cases of good reduction in one single theorem that can explain all the geometry involved, second to go far beyond the world of elliptic curves and lastly to study small ramification phenomena. 


Theorem \ref{BIG} stems from an idea of K. Kato, who in fact expects that a very general duality theorem  about the image of $K(K;G_1,G_2)/m\rightarrow H^{2}(K,G_1[m]\otimes G_2[m])$ should exist for semi-abelian $G_1,G_2$ over a $p$-adic field $K$. Such a theorem would contain our theorem \ref{BIG} and the Merkurjev-Suslin theorem as the two edge cases. In this article we start with the good reduction case, and we hope in a forthcoming paper to consider the next step, that of abelian varieties of semistable reduction.

We note that indeed all the previously known computations can be though of as special cases of theorem \ref{BIG}. We give the following example.
\begin{exmp} For the self product $E\times E$  of an elliptic curve with good ordinary reduction over $K$ such that $E[p]\subset E(K)$, Murre and Ramakrishnan (\cite{Mur/Ram}) and Raskind and Spiess (\cite{Ras/Sp}) showed that $\img(s_p)=\Z/p$. This computation seen from our perspective, corresponds to the fact that there are no non-trivial homomorphisms of finite flat group schemes over $\mathcal{O}_K$ from $\mu_p$ to $\Z/p$.
\end{exmp}

The main advantage of our approach lies in our method of proving the main theorem. Our techniques are very different than the classical approach to the study of Galois symbols, which involves explicit computations of Hilbert symbols. In this article we suggest instead the use of $p$-adic cohomology theories, to show that strong relations exist between $K$-groups and zero cycles on the one side and $p$-adic Hodge theoretic objects on the other.  An inspiration for our approach comes from the work of Kato in \cite{Kato11}, where he relates the study of the classical Galois symbol $K_2^M(K)/p^n\rightarrow H^2(K,\mu_{p^n}^{\otimes 2})$ to the work of Fontaine-Messing (\cite{Fon-Mess}). 
\subsubsection{Small ramification} When the ramification index of $K$ is small, we obtain the following corollary, based on a classical result of Raynaud (\cite{Raynaud}).
\begin{cor} When the ramification index $e$ of $K$ satisfies $e<p-1$ and the abelian varieties $A$, $B$ have good reduction, the Galois symbol $s_{p^n}$ vanishes, for every $n\geq 1$. 
\end{cor}
\subsubsection{Limit information} Not imposing any limitations on the base field $K$ allows us to pass to the limit for subgroups of finite index. When we pass to the limit for $m\geq 1$,  we get a pairing
\[\lim\limits_{\longleftarrow}H^2(K,A[m]\otimes B[m])\times\Hom_{G_K}(T(A),T(B^\star))\rightarrow\Q/\Z,\] where $T(A), T(B)$ are the total Tate modules of $A, B$ respectively. Theorem \ref{BIG} together with the classical result of Tate on $p$-divisible groups (\cite{Tate1}) give us the following corollary.
\begin{cor}\label{limit1} Let $A,B$ be abelian varieties over $K$ with good reduction. The image of the Galois symbol $K(K;A,B)\rightarrow\lim\limits_{\longleftarrow}H^2(K,A[m]\otimes B[m])$ lies on the left kernel of the pairing
\[(\star')\;\lim\limits_{\longleftarrow}H^2(K,A[m]\otimes B[m])\times\Hom_{G_K}(T(A),T(B^{\star}))\rightarrow\Q/\Z.\] Moreover, it is $p^s$-torsion, where $s$ is a nonnegative integer that depends only on the ramification index of $K$.
\end{cor}

The above corollary implies in particular that for an abelian variety $A$ of dimension $d$ over $K$ with good reduction, the  cycle map  to \'{e}tale cohomology,
\[CH_{0}(A)\rightarrow H^{2d}_{et}(A,\Z_p(1)^{\otimes d}),\] when restricted to the Albanese kernel, $T(A)$, has finite image.
 This agrees with a well-known conjecture of Colliot-Th\'{e}l\`{e}ne (\cite{CT2}), which states that for a smooth projective variety $X$ over a $p$-adic field, the group $T(X)$ is the direct sum of its maximal divisible subgroup and a finite group. So far a weaker version of this conjecture has been established by S. Saito and K. Sato (\cite{SS1}). Namely, they prove that the degree zero subgroup $A_0(X)$ is the direct sum of a coprime-to-$p$-divisible group and a finite group.

Our computations indicate that for abelian varieties with good reduction the $p$-part of the Albanese kernel also behaves well, at least with respect to the cycle map. It remains to show that the above map has at most a finite kernel. 



\subsection{Relation to the finite Bloch-Kato cohomology} 

Bloch and Kato in \cite{Bloch-Kato3} defined for a $\hat{\Z}$-module $T$ of finite rank with a continuous $G_K$-action, subgroups \[H^1_e(K,T)\subset H^1_f(K,T)\subset H^1_g(K,T)\subset H^1(K,T),\] called the exponential part, the finite part and the geometric part respectively. For example, the finite part is defined as follows,
\[H^1_f(K,T):=\ker(H^1(K,T)\longrightarrow H^1(K,B_{cris}\otimes_{\hat{\Z}}T)).\] 

When $T$ is a $G_K$-stable lattice inside a crystalline Galois representation $V$, the group $H^1_f(K,T)$ classifies crystalline extensions of $T$ by $\Z_p$. 
Moreover, Bloch and Kato proved that under the local Tate duality pairing, the exact annihilator of the exponential part is the geometric part of the dual, while the exact annihilator of the finite part is the finite part of the dual.
Note that when $T$ is the Tate module of an abelian variety $A$ over $K$, all the three subgroups are equal and they coincide with $\img[A(K)\hookrightarrow H^1(K,T)]$, which recovers the finer result of Tate for $H^1$. 

\subsection*{Some recent progress} Iovita and Marmora in \cite{Iovita/Marmora2015} suggested a more general definition of $H^i_f(K,T)$ for $i=0,1,2$ and for $T$ a $G_K$-stable lattice inside a crystalline Galois representation with small Hodge-Tate weights. Moreover, they extended the definition to finite Galois modules of the form $T'/p^n$ with $T'$ as above.  Connecting this to the work of Bloch and Kato, we have the following expectation.
\begin{expectation} For a Galois module $T$ as above, the subgroups $H^2_f(K,T)$ and $H^0_f(K,T^\star)$ are orthogonal complements to each other under Tate duality. 
\end{expectation} In fact, we believe that theorem \ref{BIG} is the first step towards this expectation, considering $T=A[p^n]\otimes B[p^n]$, for abelian varieties $A.B$ with good reduction. It is not hard to see that the orthogonal complement to the image of the generalized Galois symbol coincides with the subgroup $H^0_f(K,A[p^n]\otimes B[p^n])$. It remains to check that $\img(s_{p^n})$ is exactly $H^2_f(K,A[m]\otimes B[m])$, which will be precisely the $H^2$-analogue of Bloch and Kato's computation for abelian varieties. 

\subsection{Outline of the proof and main results} After reviewing some basic facts about local Tate duality in section 2, we give some first indications towards theorem \ref{BIG}, including proving the theorem for integers $m$ that are coprime to $p$. In the rest of the article we focus on proving the theorem for $m=p^n$. Our main tool for the proof is the interplay between zero cycles and the $K$-group $K(K;A,B)$. The key is a result of S. Saito and K. Sato (\cite{SS-14}) about zero cycles and Brauer groups. Their computation allows us to relate the orthogonal complement of the Galois symbol to a $p$-adic Hodge theoretic object, namely to a syntomic cohomology group.

 In section 3 we focus on the case of a single abelian variety $A$ over $K$. We consider the Albanese kernel $T(A)$ of $A$. It is well known that the cycle map  to \'{e}tale cohomology when restricted to $T(A)$ maps to $ H^{2}(K,\bigwedge^{2}A[p^{n}])$. Moreover, this map has been related to a Galois symbol by the current author (\cite{Gaz12015}). Using the main results of \cite{Gaz12015} and \cite{SS-14} we prove the following theorem.
\begin{thm}\label{flat2} Let $p$ be an odd prime. Let $A$ be an abelian variety over $K$ of dimension $d$ with good reduction. Let $\mathcal{A}$ be the N\'{e}ron model of $A$. Under the Tate duality pairing \[(\star\star)\;H^{2}(K,\bigwedge^{2}A[p^{n}])\times H^{0}(K,\Hom(\bigwedge^{2}A[p^{n}],\mu_{p^{n}}))\rightarrow\Z/p^{n},\]
the orthogonal complement of the image of the cycle map $T(A)\stackrel{c_{p^{n}}}{\rightarrow} H^{2}(K,\bigwedge^{2}A[p^{n}])$, is the image of the composition
$H^{2}(\mathcal{A}_{fl},\mu_{p^{n}})\rightarrow H^{2}(A_{et},\mu_{p^{n}})\rightarrow H^{0}(K,\Hom(\bigwedge^{2}A[p^{n}],\mu_{p^{n}})),$ where $H^{2}(\mathcal{A}_{fl},\mu_{p^{n}})$ is the flat cohomology of $\mathcal{A}$.
\end{thm}

In the remaining sections we connect the group $H^{2}(\mathcal{A}_{fl},\mu_{p^{n}})$ to a group of homomorphisms of finite flat group schemes. To do that, we use the work of Breuil (\cite{Bre1}, \cite{Bre2}) and Faltings (\cite{Fal}).
We first replace flat with syntomic cohomology in section 4. Then we pass to integral crystalline cohomology in the sense of Faltings (\cite{Fal}). For an abelian scheme $\mathcal{A}$ over $\Spec(\mathcal{O}_{K})$, we denote these groups by  $H^{\star}_{Br}(\mathcal{A}/S)$.  For $i<p-1$, the group $H^{i}_{Br}(\mathcal{A}/S)$ is a strongly divisible module over a specific ring $S$ equipped with divided powers.  In particular, it is a filtered module with a Frobenius action and a connection $\nabla$. The cases we are interested in are when $i=1$ and $2$. After discussing some generalities about strongly divisible modules, in section 6 we connect $H^{1}_{Br}(\mathcal{A}/S)$ to the Breuil module corresponding to the $p$-divisible group $\lim\limits_{\longrightarrow}\mathcal{A}[p^{n}]$. Moreover, we prove the following theorem.
\begin{thm} There is a natural isomorphism of filtered free modules \[\rho:\bigwedge^{2}H^{1}_{Br}(\mathcal{A}/S)\stackrel{\simeq}{\longrightarrow} H^{2}_{Br}(\mathcal{A}/S).\]
\end{thm} We note that some of these computations might be known to the experts, but we could not find any reference in the literature.

The proof of theorem \ref{BIG} is completed in section 7. We then finish the article  by discussing some applications to zero cycles and providing some specific examples when the Galois symbol is non trivial.

\subsection{Notation} In all this draft $K$ will be a finite extension of $\Q_{p}$ with ring of integers $\mathcal{O}_{K}$ and residue field $k$. We will denote by $W=W(k)$ the ring of Witt vectors over $k$. We fix an algebraic closure $\overline{K}$ of $K$. For an extension $L$ of $K$ and a variety $X$ over $K$, we will denote by $X_{L}$ its base change to $L$.

For an abelian variety $A$ and a positive integer $m$, we will denote by $A[m]$ the $m$-torsion points of $A_{\overline{K}}$.  Moreover, we will denote by $A^{\star}$ the dual abelian variety of $A$.

For an abstract abelian group $C$ we will denote by $C^{\vee}:=\Hom(C,\Q/\Z)$ and by $C^{\star}:=\Hom(C,\Q/\Z(1))$, where $\Q/\Z(1)$ is the Tate twist.

For finite flat group schemes $\mathcal{P},\mathcal{Q}$ over $\mathcal{O}_{K}$, we will denote by $\mathcal{H}om_{ffgps/\mathcal{O}_{K}}(\mathcal{P},\mathcal{Q})$ the group of homomorphisms of finite flat group schemes defined over $\Spec(\mathcal{O}_{K})$.

Throughout the article, we will be working with many different Grothendieck sites and cohomology theories. In each case we will use a subscript, for example $H^{\star}(X_{fl},-)$ for the flat cohomology. By abuse of notation, sometimes we will omit the subscript whenever we work with \'{e}tale cohomology. Moreover, when we have a morphism of Grothendieck topologies, we will write the map as a functor on the corresponding categories of abelian sheaves, for example $X_{et}\rightarrow X_{fl}$.
\subsection{Acknowledgements} I would like to express my deepest gratitude to my Ph.D thesis advisor, Pr. Kazuya Kato. There are not enough words that can express my gratitude. This project was initiated under K. Kato's supervision and guidance. In fact, Pr. Kato shared his vision and ideas with me, helped me very significantly with the outline of the proof and gave precious lectures to me about $p$-adic Hodge theory. Additionally, I am deeply grateful to my post doc mentor at the University of Michigan, Pr. Bhargav Bhatt, who was very kind and eager to help me with the crystalline cohomology computations, and moreover he  proof read my paper and gave me very useful suggestions. Furthermore, I would like to thank Pr. Christophe Breuil for a very useful discussion about his theory of Breuil modules.

This project was started while at the University of Chicago, continued while the author was visiting the SFB program on higher invariants at the University of Regensburg and completed while at the University of Michigan. I would like to thank all three institutions for their support and hospitality.
\medskip
\section{The Galois symbol and Tate duality}  In this section we review the definition of the local Tate duality and define the symbolic part of Galois cohomology. We then show some first indications towards theorem (\ref{BIG}).
\subsection{Local Tate duality}
Let $A,B$ be abelian varieties over $K$. Let $\;B^{\star}$ be the dual abelian variety of $B$.
Let $m\geq 1$ be a positive integer, possibly divisible by $p$. For the finite $G_{K}$-module $A[m]\otimes B[m]$, local Tate duality gives us a perfect pairing of finite abelian groups
\[H^{2}(K,A[m]\otimes B[m])\times H^{0}(K,\Hom(A[m]\otimes B[m],\mu_{m}))\rightarrow\Z/m.\]
We start by giving a different description of the group $H^{0}(K,\Hom(A[m]\otimes B[m],\mu_{m}))$.
\begin{lem}\label{adj} The group $H^{0}(K,\Hom(A[m]\otimes B[m],\mu_{m}))$ is canonically isomorphic to the group  $\Hom_{G_{K}}(A[m],B^{\star}[m])$.
\end{lem}
\begin{proof} First, by the definition of Galois cohomology we have
$H^{0}(K,\Hom(A[m]\otimes B[m],\mu_{m}))=(\Hom(A[m]\otimes B[m],\mu_{m})))^{G_{K}}$. Then using adjoint functors, we obtain an isomorphism  \[\Hom(A[m]\otimes B[m],\mu_{m})^{G_{K}}\stackrel{\simeq}{\longrightarrow}\Hom_{G_{K}}(A[m],\Hom(B[m],\mu_{m})).\]
Finally, the Weil pairing yields an isomorphism of $G_{K}$-modules,
$\Hom(B[m],\mu_{m})\simeq B^{\star}[m]$ from which the lemma follows.

\end{proof}
Using lemma \ref{adj}, we give a precise description of the Tate duality pairing,
\[(\star)\;\;\;H^{2}(K,A[m]\otimes B[m])\times \Hom_{G_{K}}(A[m],B^{\star}[m])\stackrel{<,>}{\longrightarrow}\Z/m.\]
Let $\{f_{\sigma,\tau}\}_{\sigma,\tau\in G}\subset A[m]\otimes B[m]$ be a 2-cocycle and $g:A[m]\rightarrow B^{\star}[m]$ be a $G_{K}$- homomorphism. Then $g$ induces a push-forward $g_{\star}(\{f_{\sigma,\tau}\})\in H^{2}(K,B^{\star}[m]\otimes B[m])$. Using the Weil pairing, $B^{\star}[m]\otimes B[m]\rightarrow \mu_{m}$, the Tate pairing $<\{f_{\sigma,\tau}\},g>$ is defined to be the class of $g_{\star}(\{f_{\sigma,\tau}\})$ in $H^{2}(K,\mu_{m})$. The latter group is isomorphic to $\Z/m$, since it coincides with the $m$-torsion points of the Brauer group, $Br(K)$, which is isomorphic to $\Q/\Z$.

From now we will refer to this Tate pairing as pairing $(\star)$.
\subsection{The Galois symbol} The Somekawa $K$-group $K(K;A,B)$ attached to the abelian varieties $A,B$ is defined as a quotient \[K(K;A,B)=(\bigoplus_{L/K}A(L)\otimes B(L))/R,\] where the sum is under all finite extensions $L$ of $K$. The relations $R$ describe a projection formula and a reciprocity relation arising from function fields of curves over $K$. In this article we won't need the exact generators of $R$.

The elements of $K(K;A,B)$ are traditionally denoted as symbols. Namely the generators of  $K(K;A,B)$ are of the form $\{a,b\}_{L/K}$, where $a\in A(L),b\in B(L)$. For a positive integer  $m\geq 1$, the Galois symbol $s_{m}$ is defined as follows,
\begin{eqnarray*}s_{m}:&&K(K;A,B)/m\rightarrow H^{2}(K,A[m]\otimes B[m])\\
&&\{a,b\}_{L/K}\rightarrow\Cor_{L/K}(\delta_{1}(a)\cup\delta_{2}(b)),
\end{eqnarray*} where $\delta_{1}$ (resp. $\delta_{2}$) is the connecting homomorphism $A(L)\rightarrow H^{1}(L,A[m])$ of the long exact sequence arising from the Kummer sequence (see subsection \ref{intro1}) of $A$ (resp. B), $\cup$ is the cup product in Galois cohomology and $\Cor_{L/K}:H^{2}(L,A[m]\otimes B[m])\rightarrow H^{2}(K,A[m]\otimes B[m])$ is the corestriction.
\begin{defn} We define the symbolic part of $H^{2}(K,A[m]\otimes B[m])$ to be the image of the symbol map $s_{m}$ and we denote it by
$H^{2}_{s}(K,A[m]\otimes B[m])$.
\end{defn}
We now restate the main theorem of this paper.
\begin{thm}\label{BIG2} Let $A$, $B$ be abelian varieties over $K$ with good reduction. Let $\mathcal{A}$, $\mathcal{B}^{\star}$ be the N\'{e}ron models of $A$ and $B^{\star}$ respectively. Under the Tate duality pairing $(\star)$, the orthogonal complement of the symbolic part $H^{2}_{s}(K,A[m]\otimes B[m]))$ is the group $\mathcal{H}om_{ffgps/\mathcal{O}_{K}}(\mathcal{A}[m],\mathcal{B}^{\star}[m])$.
\end{thm}
The very first step is to prove that the two subgroups in question annihilate each other.
\begin{lem}\label{group} Let $\mathfrak{X}$ be a smooth scheme over $\mathcal{O}_{K}$. Let $G$ be an \'{e}tale commutative group scheme. The map $H^{i}(\mathfrak{X}_{et},G)\stackrel{\rho}{\longrightarrow}
H^{i}(\mathfrak{X}_{fl},G)$ induced by the inclusion of categories $\mathfrak{X}_{et}\subset \mathfrak{X}_{fl}$ is an isomorphism.
\end{lem}
\begin{proof} This is a known result due to Grothendieck. (Theorem II.7 of Exp. III of \cite{Groth}). 

\end{proof}
\begin{prop}\label{annih} Under the Tate duality pairing $(\star)$, the groups $H^{2}_{s}(K,A[m]\otimes B[m])$ and  $\mathcal{H}om_{ffgps/\mathcal{O}_{K}}(\mathcal{A}[m],\mathcal{B}^{\star}[m])$ pair to zero.
\end{prop}
\begin{proof} Let $L/K$ be a finite extension of $K$. By the N\'{e}ron model property, we have an isomorphism $A(L)\simeq\mathcal{A}(\mathcal{O}_{L})$. Moreover, we have commutative diagrams,
\[ \xymatrix{
&  \mathcal{A}(\mathcal{O}_{L})\ar[r]\ar[d]^{\simeq} & H^{1}_{fl}(\Spec(\mathcal{O}_{L}),\mathcal{A}[m])\ar[d]\ar[r]^{\Cor_{L/K}} & H^{1}_{fl}(\Spec(\mathcal{O}_{K}),\mathcal{A}[m])\ar[d]\\
&A(L)\ar[r]^{\delta_{1}} & H^{1}(L,A[m])\ar[r]^{\Cor_{L/K}} & H^{1}(K,A[m]),
}
\] and similarly for $B$. Here we denote by $H^{1}_{fl}(\Spec(\mathcal{O}_{K}),-)$ the flat cohomology on $\Spec(\mathcal{O}_{K})$. (By flat we mean we are working over the fppf site). Notice that the two right vertical arrows are defined using the previous lemma, since we have isomorphisms  $H^1(K,A[m])\stackrel{\simeq}{\longrightarrow}H^1_{fl}(K,A[m])$ and same for $B$. Moreover, the map $\mathcal{A}(\mathcal{O}_{L})\rightarrow H^{1}_{fl}(\Spec(\mathcal{O}_{L}),\mathcal{A}[m])$ is induced by the Kummer sequence on $\mathcal{A}$, which is exact on $\Spec(\mathcal{O}_{K})_{fl}$.  Using the fact that the cup product commutes with the Corestriction map, we conclude that  $H^{2}_{s}(K,A[m]\otimes B[m])$ is contained in the image of
$H^{2}_{fl}(\Spec(\mathcal{O}_{K}),\mathcal{A}[m]\otimes\mathcal{B}[m])\rightarrow H^{2}(K,A[m]\otimes B[m])$.

Let $\{a,b\}_{L/K}\in K(K;A,B)/m$ and
$h:\mathcal{A}[m]\rightarrow\mathcal{B}^{\star}[m]$ be a morphism of finite flat group schemes over $\mathcal{O}_{K}$. Using the description of $(\star)$ described in the beginning of the section, we conclude that  $<s_{m}(\{a,b\}_{L/K}),h>\in\img[H^{2}_{fl}(\Spec(\mathcal{O}_{K}),\mu_{m})\rightarrow H^{2}(K,\mu_{m})]$. The proposition then follows by the fact that $H^{2}_{fl}(\Spec(\mathcal{O}_{K}),\mu_{m})\simeq Br(\mathcal{O}_{K})[m]=0$.

\end{proof}
\subsection{The prime-to-$p$ Case} W. Raskind and M. Spiess in \cite{Ras/Sp} (theorem 3.5) prove that when $A,B$ have good reduction, the Somekawa $K$-group $K(K;A,B)$ is $m$-divisible, for $m$ coprime to $p$. This in particular yields an equality $H^{2}_{s}(K,A[m]\otimes B[m])=0$. The next proposition shows that their result agrees with our theorem \ref{BIG2}.
\begin{prop} Let $m\geq 1$ be a positive integer with $m$ coprime to $p$. Then theorem (\ref{BIG2}) holds. In particular, we have an  equality $\mathcal{H}om_{ffgps/\mathcal{O}_{K}}(\mathcal{A}[m],\mathcal{B}^{\star}[m])=\Hom_{G_{K}}(A[m],
B^{\star}[m])$.
\end{prop}
\begin{proof}
Since both $A,B^{\star}$ have good reduction, the schemes $\mathcal{A}[m]$ and $\mathcal{B}^{\star}[m]$ are finite flat. Since $m$ is coprime to $p$ they are in fact  \'{e}tale. Let $G_{k}=\Gal(\overline{k}/k)$. The functor $\mathcal{G}\to(\mathcal{G}\times_{\mathcal{O}_K}k)(\overline{k})$ from finite \'{e}tale group schemes over $\mathcal{O}_K$ to finite continuous $G_k$-modules is faithfully flat (\cite{Tate1}, section 1.4). 
We denote by $\overline{A}:=\mathcal{A}\times_{\mathcal{O}_{K}}k$ (resp. $\overline{B^{\star}}$) the reduction of $A$ (resp. of $B^{\star}$) over the residue field. The above fact yields an isomorphism, \[\mathcal{H}om_{ffgps/\mathcal{O}_{K}}(\mathcal{A}[m],\mathcal{B}^{\star}[m])\simeq\Hom_{G_{k}}(\overline{A}[m],
\overline{B}^{\star}[m]).\]
Next, we get an isomorphism \[\Hom_{G_{k}}(\overline{A}[m],
\overline{B}^{\star}[m])\simeq\Hom_{G_{K}}(A[m],
B^{\star}[m]),\] because both $A[m],B^\star[m]$ are unramified $G_K$-modules. (\cite{Serre/Tate1968}, theorem 1). The last isomorphism completes the proof of the proposition. 

\end{proof}

\section{Zero cycles and the result of Saito-Sato} The remaining sections of this paper concentrate in proving the theorem for $m=p^{n}$, where $n$ is a positive integer. In this section we review the relation between the Somekawa $K$-group $K(K;A,B)$ and various groups of zero cycles. Using the main theorem of \cite{SS-14}, we make a first important reduction to the problem.
\subsection{Relation to zero cycles} For a smooth projective variety $X$ over $K$ we consider the group $CH_{0}(X)$ of zero cycles modulo rational equivalence on $X$. This group has a filtration \[CH_{0}(X)\supset A_{0}(X)\supset T(X),\] where the first subgroup $A_{0}(X)$ is the kernel of the degree map, $\deg:CH_{0}(X)\rightarrow\Z$, and the second subgroup $T(X)$ is the kernel of the Albanese map, $\alb_{X}:A_{0}(X)\rightarrow Alb_{X}(K)$. For more details on these definitions and a survey on problems related to zero cycles we refer to \cite{Colliot-Thelenesurvey2010}. 

There are interesting connections between the Albanese kernel of certain varieties and the Somekawa $K$-group, as the following results indicate.
\subsection*{Product of Curves} Let $C_{1},C_{2}$ be smooth complete curves over $K$ with Jacobian varieties $J_{1},J_{2}$. Raskind and Spiess in \cite{Ras/Sp} proved an isomorphism $T(C_{1}\times C_{2})\simeq K(K;J_{1},J_{2})$. In particular, if $E_{1},E_{2}$ are elliptic curves over $K$, then $T(E_{1}\times E_{2})\simeq K(K;E_{1},E_{2})$ and additionally the Galois symbol $s_{p^{n}}$ coincides with the cycle map to \'{e}tale cohomology \[c_{p^{n}}:CH_{0}(E_{1}\times E_{2})/p^{n}\rightarrow H^{4}(E_{1}\times E_{2},\mu_{p^{n}}^{\otimes 2})\] when the latter is restricted to the Albanese kernel.
\subsection*{Abelian Varieties} From now on we will be working with an abelian variety $A$ over $K$ of dimension $d$. Moreover, we assume that the prime $p$ is odd. For such a variety, the current author in \cite{Gaz12015} defined a descending filtration $\{F^{r}\}_{r\geq 0}$ of $CH_{0}(A)$ with the property \[F^{r}/F^{r+1}\otimes\Z[\frac{1}{r!}]\simeq S_{r}(K;A)\otimes\Z[\frac{1}{r!}],\] for every $r\geq 0$. (\cite{Gaz12015} th. 3.8). The group $S_{r}(K;A)$ is the quotient of the $K$-group $K(K;A,\cdots,A)$ attached to $r$ copies of $A$ by the action of the symmetric group in $r$ variables.

The case $r=2$ is of particular importance for the proof of theorem (\ref{BIG2}). The group $F^{2}$ coincides with the Albanese kernel $T(A)$ of $A$ and we have an injection $\Phi:F^{2}/F^{3}\hookrightarrow S_{2}(K;A)$ with the image of $\Phi$ containing $2S_{2}(K;A)$. (\cite{Gaz12015} prop. 3.4).
\subsection{The Hochschild-Serre spectral sequence} Let $r\geq 0$. By \cite{Mur/Ram}, lemma 2.1.1, the Hochschild-Serre spectral sequence,  \[E_{2}^{i,j}:=H^{i}(K,H^{j}(A_{\overline{K},}\Z/p^{n}(r))\Rightarrow H^{i+j}(A,\Z/p^{n}(r))\] degenerates at page 2, where $\Z/p^{n}(r):=\mu_{p^{n}}^{\otimes r}$ is the Tate twist. The proof is an easy weight argument, which holds also for $\Z/m$ coefficients. We are interested in the following two cases. \\
\underline{Filtration on $H^{2}(A,\mu_{p^{n}})$:} The spectral sequence gives a filtration $H^{2}(A,\mu_{p^{n}})=G^{0}\supset G^{1}\supset G^{2}\supset 0$ with the following graded quotients,
\begin{itemize}\item
$G^{0}/G^{1}\simeq E_{2}^{0,2}=H^{0}(K,H^{2}(A_{\overline{K}},\mu_{p^{n}}))\simeq
\Hom_{G_{K}}(\bigwedge^{2}A[p^{n}],\mu_{p^{n}})$,
\item $G^{1}/G^{2}\simeq E_{2}^{1,1}=H^{1}(K,H^{1}(A_{\overline{K}},\mu_{p^{n}}))
    =H^{1}(K,\Hom(A[p^{n}],\mu_{p^{n}}))$,
\item $G^{2}\simeq E_{2}^{2,0}=H^{2}(K,H^{0}(A_{\overline{K}},\mu_{p^{n}}))=\Z/p^{n}$.
\end{itemize}
\underline{Filtration on $H^{2d}(A,\Z/p^{n}(d))$:}
The Hochschild-Serre filtration $H^{2d}(A,\Z/p^{n}(d))=I^{0}\supset I^{1}\supset I^{2}\supset 0$ has again only three pieces, because the base field $K$ has cohomological dimension 2. The graded quotients are as follows,
\begin{itemize}\item
$I^{0}/I^{1}\simeq E_{2}^{0,2d}=H^{0}(K,H^{2d}(A_{\overline{K}},\Z/p^{n}(d)))$,
\item $I^{1}/I^{2}\simeq E_{2}^{1,2d-1}=H^{1}(K,H^{2d-1}(A_{\overline{K}},\Z/p^{n}(d)))$,
\item $I^{2}\simeq E_{2}^{2,2d-2}=H^{2}(K,H^{2d-2}(A_{\overline{K}},\Z/p^{n}(d)))$.
\end{itemize} Using Poincar\'{e} duality and the computation of $H^{i}(A_{\overline{K}},\Z/p^{n})$ (see for example \cite{Mil2}) we get a more concrete description of the graded pieces $gr^{i}(I^{\bullet})$, namely,
\begin{itemize}\item
$I^{0}/I^{1}=
H^{0}(K,\Z/p^{n})=\Z/p^{n}$,
\item $I^{1}/I^{2} =H^{1}(K,H^{2d-1}(A_{\overline{K}},\Z/p^{n}(d)))\simeq
    H^{1}(K,A[p^{n}])$,
\item $I^{2}
    =H^{2}(K,H^{2d-2}(A_{\overline{K}},\Z/p^{n}(d)))\simeq
    H^{2}(K,\bigwedge^{2}A[p^{n}])$.
\end{itemize}
\subsection{The cycle map as a Galois symbol}
Let $m$ be a positive integer. It is known that the cycle map $c_{m}:CH_{0}(A)/m\rightarrow H^{2d}(A,\Z/m(d))$ when restricted to the Albanese kernel, $T(A)$, maps to the second piece of the Hochschild-Serre filtration, namely
$c_{m}:T(A)\rightarrow H^{2}(K,\bigwedge^{2}A[m])$.

In \cite{Gaz12015}, proposition 5.2, we showed that the Galois symbol induces a map as follows, \[S_{2}(K;A)/m\rightarrow H^{2}(K,\bigwedge^{2}A[m]).\] Moreover, in section 6 of \cite{Gaz12015} we showed that the cycle map $c_{m}$ when restricted to $T(A)$ coincides with the composition
\[F^{2}/F^{3}\stackrel{\Phi}{\hookrightarrow} S_{2}(K;A)\rightarrow H^{2}(K,\bigwedge^{2}A[m]).\]
\begin{lem}\label{symbcyc} For $m=p^{n}$, with $p>2$, the image of the cycle map
$T(A)\stackrel{c_{p^{n}}}{\longrightarrow}H^{2}(K,\bigwedge^{2}A[p^{n}])$ and  that of the Galois symbol $S_{2}(K;A)\stackrel{s_{p^{n}}}{\longrightarrow} H^{2}(K,\bigwedge^{2}A[p^{n}])$ coincide.
\end{lem}
\begin{proof}
This is clear, since the group $H^{2}(K,\bigwedge^{2}A[p^{n}])$ has order a power of $p$ and the map $\Phi$ surjects onto  $2S_{2}(K;A)$, therefore the image of the composition $s_{p^{n}}\circ\Phi$ coincides with the image of $s_{p^{n}}$.

\end{proof}
In the remaining of this section our goal is to describe the orthogonal complement of $s_{p^{n}}:S_{2}(K;A)/p^{n}\rightarrow H^{2}(K,\bigwedge^{2}A[p^{n}])$ under the Tate duality pairing,
\[(\star\star)\;H^{2}(K,\bigwedge^{2}A[p^{n}])\times H^{0}(K,\Hom(\bigwedge^{2}A[p^{n}],\mu_{p^{n}}))\rightarrow\Z/p^{n},\] which is defined similarly to the pairing $(\star)$. We first need some generalities about the Brauer-Manin pairing.

\subsection*{The Brauer-Manin pairing}
Let $X$ be a smooth projective variety over $K$ of dimension $d$. The Brauer-Manin pairing is a pairing,
\[<,>_{1}:CH_{0}(X)\times Br(X)\rightarrow\Q/\Z,\] between the zero cycles on $X$ and the Brauer group of $X$. The pairing is described by evaluation of Azumaya algebras at closed points.
Restricting to the $p^{n}$-torsion points of $Br(X)$, we obtain a pairing
$<,>_{1}:CH_{0}(X)/p^{n}\times Br(X)[p^{n}]\rightarrow\Z/p^{n}.$
Moreover, S. Saito in \cite{Saito} shows the existence of a perfect pairing
\[<,>_{2}:H^{2d}(X,\Z/p^{n})\times H^{2}(X,\mu_{p^{n}})\rightarrow\Z/p^{n}.\] The Kummer sequence, $0\rightarrow\mu_{p^{n}}\rightarrow\mathbb{G}_{m}\stackrel{p^{n}}{\rightarrow}\mathbb{G}_{m}\rightarrow 0$ is an exact sequence of \'{e}tale sheaves on $X$, giving  rise to a short exact sequence
\[0\rightarrow\Pic (X)/p^{n}\stackrel{\beta}{\longrightarrow} H^{2}(X,\mu_{p^{n}})\stackrel{\varepsilon}{\longrightarrow} Br(X)[p^{n}]\rightarrow 0.\] The cycle map $c_{p^{n}}$ and the above exact sequence give a compatibility between the two pairings as follows,
\[
	\begin{tikzcd} 
& Br(X)[p^{n}]\ar{r}{h_1} & \Hom(CH_{0}(X)/p^{n},\Z/p^{n})\\
& H^{2}(X,\mu_{p^{n}})\ar{u}{\varepsilon}\ar{r}{h_2} & \Hom(H^{2d}(X,\Z/p^{n}(d)),\Z/p^n)\ar{u}{c_{p^{n}}^{\star}},
\end{tikzcd}
\] where for $i=1,2$ we denote by $h_i$ the map defined as $h_i(x)=<-,x>_i$.  For a proof of the commutativity of the above diagram we refer to \cite{Gaz12015} (lemma 6.4).

The Brauer-Manin pairing is far from being perfect. In fact it has significant kernels on both sides. We are interested in the right kernel, which is described by the following theorem of S. Saito and K. Sato (we only include the good reduction case).
\begin{thm} (\cite{SS-14}) Let $X$ be a smooth projective variety over $K$ that has a smooth proper model $\mathfrak{X}$ over $\mathcal{O}_K$. The right kernel of the Brauer-Manin pairing is precisely $Br(\mathfrak{X})$.
\end{thm} We note that the coprime-to-$p$ version of the above theorem was previously known by Colliot-Th\'{e}l\`{e}ne and S. Saito (\cite{CTSS}).
 This result is the key to prove the theorem in the next subsection.
\subsection{The orthogonal complement of the cycle map}
From now on we focus on the case of an abelian variety $A$ over $K$ with good reduction. Let $\mathcal{A}$ be the N\'{e}ron model of $A$, which is an abelian scheme over $\Spec(\mathcal{O}_{K})$.

We next consider the map $H^{2}_{fl}(\mathcal{A},\mu_{p^{n}})\rightarrow H^{2}_{et}(A,\mu_{p^{n}})$ obtained by the  composition of the map $H^{2}_{fl}(\mathcal{A},\mu_{p^{n}})\rightarrow H^{2}_{fl}(A,\mu_{p^{n}})
$  and the inverse of the map $\rho$ of lemma \ref{group} applied for the group scheme $(\mu_{p^{n}})_{K}$.
\begin{thm}\label{flat} Let $A$ be an abelian variety over $K$ of dimension $d$ with good reduction. Let $\mathcal{A}$ be the N\'{e}ron model of $A$. Under the Tate duality pairing $(\star\star)$, the orthogonal complement of the image of the cycle map $T(A)\stackrel{c_{p^{n}}}{\rightarrow} H^{2}(K,\bigwedge^{2}A[p^{n}])$, or equivalently of the Galois symbol $s_{p^{n}}$, is the image of the composition
\[H^{2}_{fl}(\mathcal{A},\mu_{p^{n}})\rightarrow H^{2}_{et}(A,\mu_{p^{n}})\twoheadrightarrow H^{0}(K,\Hom(\bigwedge^{2}A[p^{n}],\mu_{p^{n}})).\]
\end{thm}
\begin{proof} The first step is to consider the  Saito perfect pairing \[<,>_{2}:H^{2d}_{et}(A,\Z/p^{n}(d))\times H^{2}_{et}(A,\mu_{p^{n}}))\rightarrow\Z/p^{n},\] and prove that the orthogonal complement under $<,>_2$ of the image of the cycle map is  $\img(H^{2}_{fl}(\mathcal{A},\mu_{p^{n}})\rightarrow H^{2}_{et}(A,\mu_{p^{n}}))$.

Keeping the notation from our set up above, we need to show an equality \[\ker(c^{\star}_{p^{n}}\circ h_{2})=
\img(H^{2}_{fl}(\mathcal{A},\mu_{p^{n}})\rightarrow H^{2}_{et}(A,\mu_{p^{n}})).\] The compatibility between the two pairings yields an equality
$c^{\star}_{p^{n}}\circ h_{2}=h_{1}\circ\varepsilon$. Let $x$ be an element of $\ker(h_{1}\circ\varepsilon)
$. Then $\varepsilon(x)\in\ker(h_{1})$.
By the theorem of Saito and Sato (\cite{SS-14}), we conclude that $\varepsilon(x)$ belongs to $Br(\mathcal{A})[p^{n}]$. Considering the Kummer sequences on $\mathcal{A}_{fl}$ and $A_{et}$ respectively we get commutative diagrams
\[ \xymatrix{
& 0\ar[r] & \Pic(\mathcal{A})/p^{n}\ar[r]^{\beta}\ar[d]^{\iota_{1}} & H^{2}_{fl}(\mathcal{A},\mu_{p^{n}})\ar[r]^{\varepsilon}\ar[d]^{\iota_{2}}& Br(\mathcal{A})[p^{n}]\ar[r]\ar[d]^{\iota_{3}} & 0\\
& 0\ar[r] & \Pic(A)/p^{n}\ar[r]^{\beta} & H^{2}_{et}(A,\mu_{p^{n}})\ar[r]^{\varepsilon}& Br(A)[p^{n}]\ar[r] & 0.\\
}
\] By the Snake Lemma we obtain an exact sequence
\[\ker(\iota_{1})\stackrel{\beta}{\longrightarrow}\ker(\iota_{2})\stackrel{\varepsilon}{\longrightarrow}\ker(\iota_{3})\stackrel{\delta}{\longrightarrow}
\ck(\iota_{1})\stackrel{\beta}{\longrightarrow}\ck(\iota_{2})\stackrel{\varepsilon}{\longrightarrow}\ck(\iota_{3}).\]
Since $\varepsilon(x)\in\img(\iota_{3})$, we get that $x\in\img(\ck(\iota_{1})\rightarrow\ck(\iota_{2}))$. Since $A$ is the generic fiber of $\mathcal{A}$, we have a surjection $\Pic(\mathcal{A})\twoheadrightarrow\Pic(A)$. In particular $\iota_{1}$ is surjective and hence $x\in\img(\iota_{2})$, which is exactly what our first claim was.

The next step is to restrict to the Albanese kernel $T(A)$. We refine the Saito pairing using Tate duality. We consider the Hochschild-Serre filtration $G^{0}\supset G^{1}\supset G^{2}\supset 0$ and $I^{0}\supset I^{1}\supset I^{2}\supset 0$ of the groups $H^{2}_{et}(A,\mu_{p^{n}})$ and
$H^{2d}_{et}(A,\Z/p^{n}(d))$ respectively. In subsection 3.2 we computed the graded pieces of these.  We notice that for integers  $i,j\in\{0,2\}$ with $i+j=2$ the Saito pairing $<,>_{2}$ gives a pairing  \[<,>_{2}:gr^{i}(I^{\bullet})\times gr^{j}(G^{\bullet})\rightarrow\Z/p^{n},\] which is in fact a local Tate duality pairing. The case of interest for us is when $i=2$ and $j=0$, when the pairing coincides with $(\star\star)$.
We have commutative diagrams
\[ \begin{tikzcd}
&0 \ar{r}& G^{1}\ar{r}\ar{d}{h_2\simeq} & H^{2}(A,\mu_{p^{n}})\ar{r}\ar{d}{h_2\simeq} & G^{0}/G^{1}\ar{r}\ar{d}{h_2\simeq} & 0\\
&0 \ar{r}& (I^{0}/I^{2})^{\vee}\ar{r}\ar{d}{c_{p^{n}}
^{\star}} &
(H^{2d}(A,\Z/p^{n}(d)))^{\vee}\ar{r}\ar{d}{c_{p^{n}}^{\star}} & (I^{2})^{\vee}\ar{r}\ar{d}{c_{p^{n}}^{\star}} & 0\\
& & (CH_{0}(A)/T(A))^{\vee}\ar{r} & (CH_{0}(A)/p^{n})^{\vee}\ar{r}
& (F^{2}/p^{n})^{\vee} \ar{r} & 0.
\end{tikzcd}
\] We set $f_{i}:=c_{p^{n}}^{\star}\circ h_{2}$ (the map corresponding to the $i$ vertical arrow for $i=1,2,3$). We want to describe the kernel of $f_{3}$. A weaker form of the snake lemma yields an exact sequence
\[\ker (f_{2})\longrightarrow\ker(f_{3})\stackrel{\delta}{\longrightarrow}\ck(f_{1}).\]
By the previous step we get that $\ker (f_{2})=\img(H^{2}_{fl}(\mathcal{A},\mu_{p^{n}})\rightarrow H^{2}_{et}(A,\mu_{p^{n}}))$. The statement of the theorem will thus follow, if we show that $\ck(f_{1})=0$. It suffices to show that $c_{p^{n}}^{\star}:(H^{2d}_{et}(A,\Z/p^{n}(d))/I^{2})^{\vee}\rightarrow
((CH_{0}(A)/T(A))/p^{n})^{\vee} $ is surjective, or equivalently that
\[c_{p^{n}}:(CH_{0}(A)/T(A))/p^{n}\rightarrow H^{2d}_{et}(A,\Z/p^{n}(d))/I^{2}\]
is injective. This follows from the following two facts. First, we have an isomorphism induced by the degree map
\[\deg:(CH_{0}(A)/A_{0}(A))/p^{n}\rightarrow\Z/p^{n}\simeq H^{2d}_{et}(A,\Z/p^{n}(d))/I^{1}.\] Second, we have an inclusion induced by the Kummer sequence $0\longrightarrow A[p^{n}]\longrightarrow A\stackrel{p^{n}}{\longrightarrow} A\longrightarrow 0$ on $A$ which gives
$(A_{0}(A)/T(A))/p^{n}\simeq A(K)/p^{n}\hookrightarrow H^{1}(K,A[p^{n}])\simeq I^{1}/I^{2}.$

\end{proof}
\vspace{2pt}
\section{Syntomic cohomology} In this section we want to replace the flat cohomology group $H^{2}_{fl}(\mathcal{A},\mu_{p^{n}})$ of theorem \ref{flat} with syntomic cohomology. Let
$W=W(k)$ be the ring of Witt vectors over the residue field $k$ and $K_{0}=\fc(W(k))$ its field of fractions. We have a tower $\Q_{p}\subset K_{0}\subset K$ of finite extensions with $K_{0}/\Q_{p}$ unramified and $K/K_{0}$ totally ramified. Let $e=[K:K_{0}]$ be the absolute  ramification index.

\subsection{The syntomic site} For a smooth scheme $\mathfrak{X}$ over $\Spec(\mathcal{O}_{K})$ we denote by $\mathfrak{X}_{syn}$ the small syntomic site  of $\mathfrak{X}$. This is a Grothendieck site with coverings surjective families of syntomic morphisms. We review the definition of a syntomic morphism $Y\rightarrow\mathfrak{X}$.
\begin{defn} A morphism $f:Y\rightarrow\mathfrak{X}$ over $\Spec(\mathcal{O}_{K})$  is called syntomic, if it is flat, locally of finite presentation and \'{e}tale locally on $Y$ there is a factorization $Y\stackrel{i}{\hookrightarrow} Z\stackrel{h}{\rightarrow}\mathfrak{X}$ with $h$ a smooth morphism and $i$ an exact closed immersion that is transversally regular over $\mathfrak{X}$. Alternatively, $f$ is flat and locally a regular complete intersection.
\end{defn} In particular every \'{e}tale morphism is syntomic and every syntomic morphism is flat, hence we have  maps of sites $\mathfrak{X}_{et}\rightarrow\mathfrak{X}_{syn}\rightarrow\mathfrak{X}_{fl}$.
\begin{exmp} For  a scheme $\mathfrak{X}$ over $\mathcal{O}_{K}$, the Kummer sequence $0\rightarrow\mu_{p^n}\rightarrow\G_m\rightarrow\G_m\rightarrow 0$ is not exact on $\mathfrak{X}_{et}$, but it is exact on $\mathfrak{X}_{syn}$.
\end{exmp}

\begin{lem}\label{synflat} The map $j:H^{2}_{syn}(\mathfrak{X},\mu_{p^{n}})\rightarrow H^{2}_{fl}(\mathfrak{X},\mu_{p^{n}})$ is an isomorphism.
\end{lem}
\begin{proof}
The Kummer sequence \[0\rightarrow\mu_{p^{n}}\rightarrow\G_{m}\rightarrow\G_{m}\rightarrow 0\] is a short exact sequence both on $\mathfrak{X}_{fl}$ and on $\mathfrak{X}_{syn}$. It therefore induces long exact sequences fitting into commutative diagrams as follows
\[\xymatrix{
 H^{1}_{syn}(\mathfrak{X},\G_{m})\ar[r]\ar[d]& H^{1}_{syn}(\mathfrak{X},\G_{m})\ar[r]\ar[d] & H^{2}_{syn}(\mathfrak{X},\mu_{p^{n}})\ar[r]\ar[d] & H^{2}_{syn}(\mathfrak{X},\G_{m})\ar[r]\ar[d] & H^{2}_{syn}(\mathfrak{X},\G_{m})\ar[d] \\
 H^{1}_{fl}(\mathfrak{X},\G_{m})\ar[r]& H^{1}_{fl}(\mathfrak{X},\G_{m})\ar[r] & H^{2}_{fl}(\mathfrak{X},\mu_{p^{n}})\ar[r] & H^{2}_{fl}(\mathfrak{X},\G_{m})\ar[r] & H^{2}_{fl}(\mathfrak{X},\G_{m}).
\\
}\] The morphism of sites $\mathfrak{X}_{et}\stackrel{\varphi}{\longrightarrow}\mathfrak{X}_{fl}$ factors through $\mathfrak{X}_{syn}$, as $\mathfrak{X}_{et}\stackrel{\varphi'}{\longrightarrow}\mathfrak{X}_{syn}
\stackrel{\psi}{\longrightarrow}\mathfrak{X}_{fl}$. Lemma \ref{group} yields an equality $R^{j}\varphi_{\star}\G_m=0$. This in turns implies $R^{j}\varphi'_{\star}\G_m=0$, since the latter is the sheafification of the former for the syntomic topology. We therefore get isomorphisms $H^{i}_{syn}(\mathfrak{X},\G_{m})\simeq H^{i}_{fl}(\mathfrak{X},\G_{m})$. The lemma then follows by the five lemma.

\end{proof}
\begin{cor}\label{cyclic} The orthogonal complement of $\img(T(A)/p^n\stackrel{c_{p^n}}{\longrightarrow} H^2(K,\bigwedge^2A[p^n]))$ under the Tate pairing is the image of the composition
\[H^{2}_{syn}(\mathcal{A},\mu_{p^{n}})\rightarrow H^{2}_{et}(A,\mu_{p^{n}})\twoheadrightarrow H^{0}(K,\Hom(\bigwedge^{2}A[p^{n}],\mu_{p^{n}})).\]
\end{cor}

\begin{proof} This is an immediate consequence of theorem \ref{flat} and lemma \ref{synflat}.

\end{proof}
\subsection{The syntomic sheaf $S_{n}(r)$} Let $\mathfrak{X}$ be a syntomic scheme over $W(k)$. For an integer $n\geq 1$ we denote $\mathfrak{X}_n:=\mathfrak{X}\times\Z/p^n$.  We recall that for  such a scheme $\mathfrak{X}$ and for an integer $r\leq p-1$ we can define the following syntomic sheaves. First, the structure sheaf $\mathcal{O}_{\mathfrak{X}_n}$ of the absolute crystalline site (this means over $W_{n}=W\otimes\Z/p^{n}$). Second, the sheaves $\mathcal{O}_{n}^{cris},$ $\mathcal{J}_{n}$ and $\mathcal{J}_{n}^{[r]}$, defined as follows. 
\[\mathcal{O}_{n}^{cris}(\mathfrak{X})=\mathcal{O}_{n}^{cris}(\mathfrak{X}_1)=H^{0}_{cris}
((\mathfrak{X}_{1}/W_n)_{cris},\mathcal{O}_{\mathfrak{X}_{1}/W_n}),\;\;\;\;\;
\mathcal{J}_{n}=\ker(\mathcal{O}_{\mathfrak{X}_n}\to\mathcal{O}_{n}^{cris}),\] and $\mathcal{J}_{n}^{[r]}$ is the $r^{\text{th}}$ divided power of $\mathcal{J}_{n}$. Note that  $(\mathfrak{X}_{1}/W_n)_{cris}$ is the cristalline site of the special fiber, $\mathfrak{X}_{1}$. For $r\leq 0$ we define $\mathcal{J}_{\mathfrak{X}_{n}}^{[r]}=\mathcal{O}_{\mathfrak{X}_{n}}$. It is known (\cite{Fon-Mess}, II.1.3) that the presheaves $\mathcal{O}_{n}^{cris}$ and $\mathcal{J}_{n}^{[r]}$ are sheaves on $\mathfrak{X}_{n,syn}$. There is a Frobenius endomorphism   $\phi:\mathcal{O}_{n}^{cris}\rightarrow \mathcal{O}_{n}^{cris}$. If for example there exists a global embedding $\mathfrak{X}\subset\mathcal{Z}$, where $\mathcal{Z}$ is smooth over $W$, and with a compatible system of liftings of Frobenius $\{F_n:\mathcal{Z}_n\rightarrow\mathcal{Z}_n\}$, then the Frobenius $\phi$ is induced by the Frobenius on $\mathcal{Z}_n$ and the PD-envelope structure on $\mathcal{O}_n^{cris}(\mathfrak{X})$. This Frobenius has the property $\phi(\mathcal{J}_{n}^{[r]})\subset p^{r}\mathcal{O}_{n}$. The assumption $r\leq p-1$ is necessary for this last property. To make it work for larger values of $r$, a modification of $\mathcal{J}_{n}^{[r]}$ is needed (see \cite{Ertl/Niziol}). One can check that $\mathcal{J}_{n}^{[r]}$ is flat over $\Z/p^n$ and we have an isomorphism $\mathcal{J}_{n+1}^{[r]}\otimes\Z/p^n\simeq\mathcal{J}_{n}^{[r]}$. This allows us to define the divided Frobenius map, $\phi_{r}="p^{-r}\phi":\mathcal{J}_{n}^{[r]}\rightarrow\mathcal{O}_{n}^{cris}$. We define the syntomic sheaf $S_{n}(r)$ as the mapping cone,
\[S_{n}(r):=\Cone(\mathcal{J}_{n}^{[r]}\stackrel{1-\phi_{r}}{\longrightarrow}
\mathcal{O}_{n}^{cris})[-1].\] Alternatively, we have a short exact sequence of complexes
$0\rightarrow S_{n}(r)\rightarrow \mathcal{J}_{n}^{[r]}\stackrel{1-\phi_{r}}{\longrightarrow}
\mathcal{O}_{n}^{cris}\rightarrow 0,$ and therefore as a complex, $ S_{n}(r)=\ker(\mathcal{J}_{n}^{[r]}\stackrel{1-\phi_{r}}{\longrightarrow}\mathcal{O}_{n}^{cris})$ (the kernel being in the category of complexes).
It is known (\cite{Fon-Mess}, III.3.2) that for $r=1$ the syntomic sheaf $S_{n}(1)$ is isomorphic to $\mu_{p^{n}}$.

\subsection{Computing syntomic cohomology using differential forms}
Instead of using the syntomic site, we usually compute syntomic cohomology using de Rham complexes. In particular, assume first that there exists a closed embedding $\mathfrak{X}\hookrightarrow Z_{0}$ where $Z_{0}$ is smooth over $W(k)$ endowed with a compatible system of liftings of Frobenius $\{F_{n}:Z_{0,n}\rightarrow Z_{0,n}\}$. For every $n\geq 1$ we have embedding $\mathfrak{X}_{n}\hookrightarrow Z_{0,n}$ and therefore a surjection of \'{e}tale sheaves $\mathcal{O}_{Z_{0,n}}\twoheadrightarrow \mathcal{O}_{\mathfrak{X}_{n}}$. We consider the PD-envelope
$D_{ Z_{0,n}}(\mathfrak{X}_{n})$ of $\mathfrak{X}$ in $Z_{0,n}$. Let $J_{\mathfrak{X}_{n}\subset Z_{0,n}}$ be the PD-ideal generated by the kernel, $\ker(\mathcal{O}_{Z_{0,n}}\twoheadrightarrow \mathcal{O}_{\mathfrak{X}_{n}})$, and for $0<r\leq p-1$, let $J_{\mathfrak{X}_{n}\subset Z_{0,n}}^{[r]}$ be the $r$'th divided power of $J_{\mathfrak{X}_{n}\subset Z_{0,n}}$. We consider the de Rham complexes of sheaves on $\mathfrak{X}_{et}$,
\[ \xymatrix{
 D_{ Z_{0,n}}(\mathfrak{X}_{n})\ar[r] & D_{ Z_{0,n}}(\mathfrak{X}_{n})\bigotimes_{\mathcal{O}_{Z_{0,n}}}\Omega^{1}_{Z_{0,n}/W_{n}}\ar[r]^{d} & D_{ Z_{0,n}}(\mathfrak{X}_{n})\otimes_{\mathcal{O}_{Z_{0,n}}}\Omega^{2}_{Z_{0,n}/W_{n}}\ar[r] &\cdots
}
\] and
\[ \xymatrix{
 \mathcal{J}_{\mathfrak{X}_{n}\subset Z_{0,n}}^{[r]}\ar[r] & \mathcal{J}^{[r-1]}_{\mathfrak{X}_{n}\subset Z_{0,n}}\otimes_{\mathcal{O}_{Z_{0}}}\Omega^{1}_{Z_{0}/W}\ar[r]^{d} & \mathcal{J}_{\mathfrak{X}_{n}\subset Z_{0,n}}^{[r-2]}\otimes_{\mathcal{O}_{Z_{0}}}\Omega^{2}_{Z_{0}/W}\ar[r] &\cdots
}
\] From now on we will denote them as $\mathbb{O}_{\mathfrak{X}_{n}\subset Z_{0,n}}:=D_{ Z_{0,n}}(\mathfrak{X}_{n})\otimes_{\mathcal{O}_{Z_{0,n}}}\Omega^{\bullet}_{Z_{0,n}/W_{n}}$ and
$\mathbb{J}^{[r]}_{\mathfrak{X}_{n}\subset Z_{0,n}}:=\mathcal{J}^{[r-\bullet]}_{\mathfrak{X}_{n}\subset Z_{0,n}}\otimes_{\mathcal{O}_{Z_{0,n}}}\Omega^{\bullet}_{Z_{0,n}/W_{n}}$.
Furthermore for $0\leq r\leq p-1$ we consider the complex of \'{e}tale sheaves
\[S_{n}(r)_{\mathfrak{X},Z_{0}}:=\Cone(\mathbb{J}^{[r]}_{\mathfrak{X}_{n}\subset Z_{0,n}}\stackrel{1-\phi_{r}}{\longrightarrow}\mathbb{O}_{\mathfrak{X}_{n}\subset Z_{0,n}})
[-1],\] where for a definition of the Frobenius map $\phi_r$ we refer to \cite{Kato5}, cor. 1.5. 
This complex is up to canonical quasi-isomorphism independent of the choice of the embedding $X_{0}\hookrightarrow Z_{0}$ and the liftings of the Frobenius, so it gives a well defined object in the derived category $D^{+}(\mathfrak{X}_{et},\Z/p^{n})$ of bounded above complexes. From now on we will omit the subscript $Z_{0}$ from the notation and write simply $S_{n}(r)_{\mathfrak{X}}$.

In general, if such a global embedding does not exist, we find local embeddings and everything glues well in the derived category. It is a theorem of K. Kato (th. 4.3 in \cite{Kato5}) that the following equality holds for syntomic cohomology
\[H^{i}(\mathfrak{X}_{syn},S_{n}(r))\simeq H^{i}(\mathfrak{X}_{et},S_{n}(r)_{\mathfrak{X}}).\]

\begin{rem} We will very soon focus on the case $r=1$. We note that the complex $S_{n}(1)_{\mathfrak{X}}$
is quasi-isomorphic to the complex
 \[ \xymatrix{ \mathcal{J}^{[1]}_{\mathfrak{X}_{n}\subset Z_{0,n}}\ar[r]\ar[d]^{1-\varphi_{1}} & D_{ Z_{0,n}}(\mathfrak{X}_{n})\otimes_{\mathcal{O}_{Z_{0}}}\Omega^{1}_{Z_{0,n}/W_{n}}\ar[d]^{1-\varphi_{1}}\\
D_{ Z_{0,n}}(\mathfrak{X}_{n})\ar[r] & D_{ Z_{0,n}}(\mathfrak{X}_{n})\otimes_{\mathcal{O}_{Z_{0,n}}}\Omega^{1}_{Z_{0}/W_{n}}
}.
\] For, the map $\varphi_{1}:D_{ Z_{0,n}}(\mathfrak{X}_{n})\otimes_{\mathcal{O}_{Z_{0,n}}}\Omega^{i}_{Z_{0,n}/W_{n}}\rightarrow D_{ Z_{0,n}}(\mathfrak{X}_{n})\otimes_{\mathcal{O}_{Z_{0,n}}}\Omega^{i}_{Z_{0,n}/W_{n}}$ is divisible by $p$ for $i\geq 2$. This is because, $\phi_{1}$ on $\Omega^{i}_{Z_{0,n}/W_{n}}$ has the property  $\phi_{1}=p^{i-1}\phi_{i}$ and hence  $\phi_{1}$ on
$D_{ Z_{0,n}}(\mathfrak{X}_{n})\otimes_{\mathcal{O}_{Z_{0,n}}}\Omega^{i}_{Z_{0,n}/W_{n}}$ becomes $p\otimes p^{i-2}\phi_{i}$. The remark then follows, since both the above complexes  are complexes of sheaves on $\mathfrak{X}_{n}$ which has the same underlying topological space as $\mathfrak{X}_{1}$.
\end{rem}
\vspace{1pt}
\section{The categories of Breuil Modules}
The next step is to pass from the syntomic cohomology group to crystalline cohomology. We first need some background on Breuil modules.

We fix a uniformizer $\pi$ of $K$ and let $E(u)$ be its minimal polynomial (an Eisenstein polynomial of degree $e$). We will denote by $\sigma:W(k)\rightarrow W(k)$ the absolute Frobenius map and let $g:W[u]\rightarrow \mathcal{O}_{K}$ be the surjection sending $E(u)$ to $0$ (or $u$ to $\pi$).
\subsection{The ring $S$} We consider the ring $S:=W[[u]]^{PD}$, which is the $p$-adic  completion of the ring of divided powers of the polynomial ring $W[u]$ corresponding to the surjection $g$. The ring $S$ is endowed with the following structures.
\begin{itemize}\item A Frobenius operator, which is the unique continuous, $\sigma-$semilinear application $S\stackrel{\phi}{\longrightarrow} S$ that sends $u$ to $u^{p}$.
\item A decreasing filtration $\{\Fil^{i}S\}_{i\geq 0}$, where $\Fil^{i}S$ is the $p$-adic completion of the ideal generated by $\{\displaystyle \frac{E(u)^{n}}{n!},\;n\geq i\}$. In particular, $\Fil^{1}S$ is the $p$-adic completion of the kernel of $g$ and for $i\leq p-1$ the ideals $\Fil^{i}S$ are just powers of $\Fil^{1}S$.
\item A continuous linear derivation $N:S\rightarrow S$ such that $N(u)=-u$.
\end{itemize}
The Frobenius, the filtration and the derivation satisfy the following important properties.
\begin{enumerate}[(a)]\item For $0\leq i\leq p-1$, $\phi(\Fil^{i}S)\subset p^{i}S$, so we can define $\phi_{i}:\Fil^{i}(S)\rightarrow S$, where  $\displaystyle\phi_{i}:=\frac{\phi}{p^{i}}|_{\Fil^{i}}$. Moreover, we denote by $c=\phi_{1}(E(u))$. It is well known that $c$ is a unit of $S$.
\item $N\phi=p\phi N$
\item $N(\Fil^{i}S)\subset \Fil^{i-1}S$, for $i\geq 1$.
\end{enumerate}
\subsection{Strongly divisible and filtered free modules} The main references for this section are \cite{Bre1} and \cite{Bre2}. We fix a positive integer $1\leq r\leq p-1$.
\begin{defn}\label{mod1} Let $'\Mod^{\phi,N}_{[0,r]}$ be the category  with objects quadruples $(M, \Fil^{r}M,\phi_{r}, N)$, where
\begin{enumerate}\item
$M$ is an $S$-module
\item $\Fil^{r}M$ is an $S$-submodule such that ($\Fil^{r}S)\cdot M\subset \Fil^{r}M$
\item $\phi_{r}:\Fil^{r}M\rightarrow M$ is a $\phi$-semilinear map such that $\phi_{r}(sx)=c^{-r}\phi_{r}(s)\phi_{r}(E(u)^{r}x)$, for every $s\in \Fil^{r}S$, $x\in M$.
\item $N:M\rightarrow M$ is $W$-linear endomorphism satisfying $N(sx)=N(s)x+sN(x)$ for $s\in S, x\in M$.
\end{enumerate}
The morphisms in $'\Mod^{\phi,N}_{[0,r]}$ are $S$-linear maps that preserve $\Fil^{r}$ and commute with $\phi_{r}, N$.
\end{defn} We denote by $'\Mod^{\phi}_{[0,r]}$ the category which forgets $N$ in the definition of $'\Mod^{\phi,N}_{[0,r]}$.
\begin{rem}
In the above definition, we can alternatively define the objects of $'\Mod^{\phi,N}_{[0,r]}$ as quadruples $(M, \Fil^{r}M,\phi_{0}, N)$, with $\phi_{0}:M\rightarrow M$ a $\sigma-$semilinear map satisfying $\phi_{0}(\Fil^{r}M)\subset p^{r}M$. Indeed, if such a $\phi_{0}$ exists, we can define $\displaystyle\phi_{r}=\frac{\phi_{0}}{p^{r}}|_{\Fil^{r}M}$. Conversely, if we are given $\phi_{r}$ as in definition (\ref{mod1}), we can define $\phi_{0}:M\rightarrow M$ by $\phi_{0}(x)=c^{-r}\phi_{r}(E(u)^{r}x)$.
\end{rem}
We next consider the full subcategory of $'\Mod^{\phi,N}_{[0,r]}$ of strongly divisible modules of weight $\leq r$. The following definition is theorem 2.2.3. in \cite{Bre2}.
\begin{defn}\label{mod2} The category $\Mod^{\phi,N}_{[0,r]}$ (resp. $\Mod^{\phi}_{[0,r]}$) of strongly divisible modules of weight $\leq r$ is the full subcategory of $'\Mod^{\phi,N}_{[0,r]}$ (resp. $'\Mod^{\phi}_{[0,r]}$)  of objects $(M, \Fil^{r}M,\phi_{r}, N)$ satisfying the following conditions:
\begin{enumerate}\item $M$ is a free $S$-module of finite rank.
\item $\Fil^{r}M\cap pM=p\Fil^{r}M$
\item $\phi_{r}(\Fil^{r}M)$ generates $M$ as an $S$-module
\item $E(u)N(\Fil^{r}M)\subset \Fil^{r}M$ (Griffith's transversality)
\item The following diagram is commutative,
\[\xymatrix{
& \Fil^{r}M\ar[r]^{\phi_{r}}\ar[d]^{E(u)N} & M \ar[d]^{cN}\\
& \Fil^{r}M\ar[r]^{\phi_{r}} & M.
\\
}\]
\end{enumerate}
\end{defn}
\begin{conv} If $(M,\Fil^{r}M,\phi_{r},N)$ is a strongly divisible module, by abuse of notation we will denote this module by $M$.
\end{conv}
\begin{rem} (1.) If $r+1\leq p-1$, there is a full faithful embedding $\Mod^{\phi,N}_{[0,r]}\subset \Mod^{\phi,N}_{[0,r+1]}$. In fact, if $M$ is a strongly divisible module of weight $r$, then it becomes a module of weight $r+1$ if we define $\Fil^{r+1}(M)=\Fil^{1}S\cdot\Fil^{r}M+\Fil^{r+1}S\cdot M$. \\
(2.) Because strongly divisible modules are required to be free over $S$, and hence torsion-free, condition (2) of definition (\ref{mod2}) is equivalent to requiring $M/\Fil^{r}M$ to have no $p$-torsion.
\end{rem}
\begin{exmp} For $0\leq r\leq p-1$ we denote by $\mathbf{1}_{r}$ the rank one free object of $\Mod_{[0,r]}^{\phi,N}$, namely the module $(S,\Fil^{r}(S),\phi_{r},N)$. Moreover, we define the $-r$-twist of $\mathbf{1}_{r}$ to be the filtered free module $\mathbf{1}_{r}[-r]$, which has $S$ as the underlying $S$-module, $\Fil^{r}(\mathbf{1}_{r}[-r])=S$, $\phi_{r}(1)=1$ and $N=0$.
\end{exmp}
We have the following straightforward lemma.
\begin{lem}\label{twist} Let $M$ be a strongly divisible module of weight $1\leq r\leq p-2$. There is a canonical isomorphism
\[\Hom_{S,\Fil^{r},\phi_{r},N}(\mathbf{1}_{1}[-1],M)\simeq\{e\in\Fil^{1}(M)/\phi_{1}(e)=e, N(e)=0\}.\]
\end{lem}
\begin{proof} The proof follows directly from the definitions, after we observe that a $S$-linear morphism $f:\mathbf{1}_{1}[-1]\rightarrow M$ is completely determined by $f(1)=e$ and for this to preserve the filtration, the Frobenius and the derivation, it is a necessary and sufficient condition that $f(1)\in\Fil^{1}M$ such that $\phi_{1}(e)=0$ and $N(e)=0$.

\end{proof}
\subsection*{Functors to Galois representations} The category $\Mod^{\phi,N}_{[0,r]}$ is endowed with a contravariant functor, $T_{st}^{\star}:\Mod^{\phi,N}_{[0,r]}\rightarrow \Z_{p}[G_{K}]-\Mod$. The functor $T_{st}^{\star}$ assigns to a strongly divisible module $M$ the following $\Z_{p}[G_{K}]-$module,
\[T_{st}^{\star}=\Hom_{S,\Fil^{\cdot},\phi,N}(M,\widehat{A}_{st}).\]
Here $\widehat{A}_{st}$ is a certain period ring. We will not need the precise definition of $\widehat{A}_{st}$ in what follows, so we do not review its definition. For more details see for example \cite{Breuil10}.

The result that is mostly useful to us is that $T_{st}^{\star}$ becomes fully faithful, if we restrict to the smaller category of filtered-free modules. We review the definition of this category below.
\begin{defn} A strongly divisible module $(M,\Fil^{r}M,\phi_{r}, N)\in\Mod^{\phi,N}_{[0,r]}$ is called filtered-free if there exists a basis $\{e_{1},\cdots,e_{d}\}$ of $M$ as an $S$-module such that there exist integers $0\leq r_{1}\leq\cdots\leq r_{d}\leq r$ for $1\leq i\leq d$ such that
\[\Fil^{r}M=\bigoplus_{i=1}^{d}E(u)^{r_{i}}Se_{i}+(\Fil^{r}S)M.\] If such a basis exists, we call it an adapted basis for $M$ with corresponding integer sequence $\{0\leq r_{1}\leq\cdots\leq r_{d}\leq r\}$.
\end{defn}
We note that most of the strongly divisible modules are not filtered free. But all modules of weight $r\leq 1$ are filtered free. Moreover, those strongly divisible modules which arise from crystalline cohomology are filtered free as well, as we shall see in the next section.
\begin{rem} We note that the condition $\displaystyle\Fil^{r}M=\bigoplus_{i=1}^{d}E(u)^{r_{i}}Se_{i}+(\Fil^{r}S)M$ is equivalent to $\displaystyle\Fil^{r}M=\bigoplus_{i=1}^{d}(\Fil^{r_{i}}S)e_{i}$.
\end{rem}
\begin{lem}\label{phi-basis} If $\{e_{1},\cdots,e_{r}\}$ is an adapted basis for $M$, then the vectors $\{\phi_{r}(E(u)^{r_{i}}e_{i})\}$ form another $S$-basis of $M$.
\end{lem}
\begin{proof} The case $r=1$ was proved by Caruso in \cite{Caruso} (in the proof of lemma 2.1.1). The general case is similar. Namely, using Nakayama lemma we can show that the vectors $\{\phi_{r}(E(u)^{r_{i}}e_{i})\}$ span $M$ and they have the correct cardinality.

\end{proof}
\begin{rem} We note that Faltings in \cite{Fal} defines filtered free modules in a slightly different way. Namely, he defines the category $\mathcal{MF}_{[0,r]}(\mathcal{O}_{K})$ to have objects filtered $S$-modules $M$ having a filtration $\{\Fil^{i}M\}_{0\leq i\leq r}$, a connection $\nabla$ and a Frobenius $\phi$ that  satisfy the following properties:
\begin{enumerate}\item $M$ is a free $S$-module having a basis $\{e_{1},\cdots,e_{r}\}$ such that for $j\leq r$ there exists integers $0\leq q_{i}\leq r$ such that $\displaystyle\Fil^{j}M=\bigoplus_{i=1}^{d} \Fil^{q_{i}}(S)e_{i}$.
\item The connection $\nabla$ satisfies Griffiths transversality, namely
\[\nabla(\frac{\partial}{\partial u})(\Fil^{i}(M))\subset \Fil^{i-1}(M).\]
\item $\phi:M\rightarrow M$ is a $\sigma$-semilinear endomorphism of $M$ such that $\phi|_{\Fil^{i}(M)}\subset p^{i}M$ and $\nabla\phi_{r}=\phi_{r-1}\nabla$. Moreover the elements $\phi_{r_{i}}(e_{i})$ form a new basis of $M$, where $r_{i}=\min\{j\leq r:e_{i}\in\Fil^{j}(M)\}$ and $\displaystyle\phi_{r_{i}}=\frac{\phi}{p^{r_{i}}}$.
\end{enumerate} The morphisms in $\mathcal{MF}_{[0,r]}(\mathcal{O}_{K})$ are $S$-linear maps that preserve the filtration and commute with $\phi$ and $\nabla$.

We can see that the two definitions are equivalent. Clearly every filtered free module in the sense of Faltings is filtered free in the sense of Breuil. For the converse, let $(M,\Fil^{r}(M),\phi_{r},N)$ be a filtered free Breuil module having an adapted basis $\{e_{1},\cdots,e_{r}\}$ with corresponding integer sequence $\{0\leq r_{1}\leq\cdots\leq r_{d}\leq r\}$. Then we can define for $i\leq r$ the $S$-submodule $\displaystyle\Fil^{i}(M)=\{x\in M:E(u)^{r-i}x\in\Fil^{r}M\}$. The condition that $\{\phi_{r}(E(u)^{r_{i}}e_{i})\}$ form a new basis is equivalent to $\{\phi_{r-r_{i}}(e_{i})\}$ to form a new basis. Finally, the transversality condition (2) of Faltings is equivalent to condition (4) in definition \ref{mod2}. From now one we will use mostly Breuil's definition, but we will be interchanging between the two definitions without mentioning it.
\end{rem}

\subsection{Tensor product of filtered free modules} In this section we define the tensor product of two filtered free modules and prove some elementary properties of this product.
\begin{defn}\label{tensor} Let $0\leq r,m\leq p-1$ be two integers such that $r+m\leq p-1$. Let $M_{1}\in\mathcal{MF}_{[0,r]}(\mathcal{O}_{K})$, $M_{2}\in\mathcal{MF}_{[0,m]}(\mathcal{O}_{K})$ be two filtered-free modules of weight $\leq r$ and $\leq m$ respectively.  We define $M_{1}\otimes M_{2}$ to be the object of $\mathcal{MF}_{[0,r+m]}(\mathcal{O}_{K})$ with
\begin{enumerate}\item $M_{1}\otimes_{S} M_{2}$ as the underlying $S$-module
\item $\Fil^{n}(M_{1}\otimes M_{2})=\sum_{i+j=n}\img[\Fil^{i}M_{1}\otimes\Fil^{j}M_{2}\rightarrow M_{1}\otimes M_{2}]+(\Fil^{n}S)(M_{1}\otimes M_{2})$, for $n\leq r+m$.
\item $\phi_{0}:M_{1}\otimes M_{2}\rightarrow M_{1}\otimes M_{2}$ to be $\phi_{0}=\phi_{0,M_{1}}\otimes\phi_{0,M_{2}}$
\item $N:M_{1}\otimes M_{2}\rightarrow M_{1}\otimes M_{2}$ by $N(m\otimes n)=N(m)\otimes n+m\otimes N(n)$.
\end{enumerate}
\end{defn}

It is straightforward that $M_{1}\otimes_{S}M_{2}$ is indeed a filtered-free module, after we observe that if $\{e_{1},\cdots,e_{d_{1}}\}$ and $\{w_{1},\cdots,w_{d_{2}}\}$ are adapted basis for $M_{1}$, $M_{2}$ respectively with corresponding sequences $\{0\leq r_{1}\leq\cdots r_{d_{1}}\leq r\}$ and $\{0\leq m_{1}\leq\cdots m_{d_{2}}\leq m\}$, then $\{e_{i}\otimes w_{j}\}$ becomes an adapted basis for $M_{1}\otimes M_{2}$ after maybe some reordering of the base elements.
\begin{rem} One could try extend the above definition to any two strongly divisible modules. The issue that arises is that it is not easy to verify the property $\Fil^{r+m}(M_{1}\otimes M_{2})\cap p(M_{1}\otimes M_{2})=p\Fil^{r+m}(M_{1}\otimes M_{2})$ without the filtered free hypothesis. In this draft all the modules we are interested in will be filtered free, so we treat only this particular case.
\end{rem}
\subsection{The dual module} In this subsection we forget the monodromy operator $N$ and we work in the subcategory of $\Mod_{[0,r]}^{\phi}$ of filtered free modules.

Caruso in \cite{Caruso} defined for a Breuil  module $M$ of weight 1, a dual module $M^{\star}$. His definition immediately generalizes to all filtered free modules. We review the definition here.
\begin{defn} Let $(M,\Fil^{r}M,\phi_{r})$ be a filtered-free module of weight $r\leq p-1$. Let $\{e_{1},\cdots,e_{d}\}$ be a basis adapted for $M$ with corresponding integer sequence $\{0\leq r_{1}\leq\cdots\leq r_{d}\leq r\}$. The dual module $(M^{\star},\Fil^{r}M^{\star},\phi_{r}^{\star})$ is defined as follows:
\begin{enumerate}\item $M^{\star}=\Hom_{S}(M,S)$
\item $\Fil^{r}(M^{\star})=\{f\in M^{\star}:f(\Fil^{r}M)\subset\Fil^{r}S\}$
\item $\phi_{r}^{\star}:\Fil^{r}(M^{\star})\rightarrow M^{\star}$ maps a function $f\in\Hom_{S}(M,S)$ to the function   $\phi_{r}^{\star}(f)$ which is defined on the basis $\{\phi_{r}(E(T)^{r_{1}}e_{1}),\cdots,\phi_{r}(E(T)^{r_{d}}e_{d})\}$ by $\phi_{r}^{\star}(f)(\phi_{r}(E(T)^{r_{i}}e_{i}))=\phi_{r}(f(E(T)^{r_{i}}e_{i}))$.
\end{enumerate}
\end{defn}
\begin{prop}(prop. 2.2.5 in \cite{Caruso}) The functor $M\rightarrow M^{\star}$ induces an exact anti-equivalence on $\Mod_{[0,r]}^{\phi}$ and $(M^{\star})^{\star}=M$.
\end{prop}
We will need the following linear algebra lemma, which gives a compatibility between tensor product and dual module.

\begin{lem}\label{main} Let $M,M'$ be filtered-free modules of weight 1. There is a canonical isomorphism
\[\Hom_{\Mod_{[0,2]}^{\phi}}(\mathbf{1}_{1}[-1],M\otimes M')\simeq \Hom_{\Mod_{[0,1]}^{\phi}}(M^{\star},M').\]
\end{lem}
\begin{proof} This is proved in the appendix. 

\end{proof}

\subsection{Torsion modules} In this subsection we discuss torsion Breuil modules. We will very soon focus on the special case of torsion modules which are of the form $\displaystyle M=\frac{\widetilde{M}}{p^{n}}$, where $n\geq 1$ is a positive integer and $\widetilde{M}$ is a strongly divisible module. For a positive integer $n\geq 1$, we denote by $S_{n}$  the ring $S/p^{n}S$. We start with the following definition (page 7 in \cite{Liu}).
\begin{defn}\label{torsion1} We consider the full subcategory $\Mod \FI^{\phi,N}_{[0,r]}$ of $'\Mod^{\phi,N}_{[0,r]}$, which has objects $(M,\Fil^{r}M,\phi_{r},N)$ such that
\begin{enumerate}\item As an $S$-module $M$ is isomorphic to $\displaystyle\bigoplus_{i\in I}\frac{S}{p^{n_{i}}S}$, where $I$ is a finite subset of $\mathbb{N}$ and $n_{i}$ are positive integers.
\item $\phi_{r}(\Fil^{r}M)$ generates $M$ as an $S$-module.
\end{enumerate}
\end{defn}
We denote by $\Mod \FI^{\phi}_{[0,r]}$ the category which forgets $N$ in the definition (\ref{torsion1}).

Let $(\widetilde{M},\Fil^{r}\widetilde{M},\tilde{\phi}_{r},\widetilde{N})$ be a strongly divisible module of weight $r$. For every $n\geq 1$, $\widetilde{M}$ induces an object $(M,\Fil^{r}M,\phi_{r},N)$ of $\Mod \FI^{\phi,N}_{[0,r]}$, by setting $M=\widetilde{M}/p^{n}\widetilde{M}$, $\Fil^{r}M=\Fil^{r}\widetilde{M}/p^{n}\Fil^{r}\widetilde{M}$ and $\phi_{r}, N$ the maps induced by $\tilde{\phi},\widetilde{N}$.
\begin{exmp} Let $n\geq 1$. We denote by $\mathbf{1}_{1}^{n}$ the weight one module $(S_{n},\Fil^{1}(S_{n}),\phi_{1},N)$. Moreover, we denote by $\mathbf{1}_{1}^{n}[-1]$ the torsion module that has $S_{n}$ as the underlying $S$-module, $\Fil^{1}(\mathbf{1}_{1}^{n}[-1])=S_{n}$, $\phi_{1}(1)=1$ and $N=0$.
\end{exmp}

\subsection*{Relation to finite flat group schemes} We consider the category $\Mod\FI_{1}^{\phi}$, which is the category of torsion modules of weight exactly one. It is a very important theorem of Breuil (\cite{Bre1}) that this category is anti-equivalent to the category of commutative finite flat group schemes over the ring of integers $\mathcal{O}_{K}$ of $K$. In particular, we review the following theorem.
\begin{thm}\label{Breuil}(\cite{Bre1}) There is an anti-equivalence of categories between $\Mod\FI_{1}^{\phi}$ and the category of commutative finite flat group schemes $G$ over $\mathcal{O}_{K}$ such that $G=\ker(p^{n}_{G})$ for some $n\geq 1$ and $\ker(p^{n}_{G})$ is flat over $\mathcal{O}_{K}$ for all $n\geq 1$ ($p^{n}_{G}$ is the multiplication by $p^{n}$ on $G$).
\end{thm} We denote by $\Mod$ the functor which sends a finite flat group scheme as in theorem (\ref{Breuil}) to its corresponding torsion module and by $\Gp$ its weak inverse, which sends a torsion module of weight one to its corresponding group scheme.
\begin{exmp} For the $p^{n}$ torsion module $\mathbf{1}^{n}_{1}$ we have an isomorphism $\Gp(\mathbf{1}^{n}_{1})\simeq(\Z/p^{n})_{\mathcal{O}_{K}}$, while for the dual module $\mathbf{1}^{n}_{1}[-1]$ we have
$\Gp(\mathbf{1}^{n}_{1}[-1])\simeq(\mu_{p^{n}})_{\mathcal{O}_{K}}$.
\end{exmp}
For weight one torsion modules, Caruso (\cite{Caruso}) extended the definition of a dual module.
The previous example suggests that the functors $\Mod$ and $\Gp$ commute with taking dual objects. This is indeed true due to the following theorem of Caruso (\cite{Caruso}, section 3).
\begin{thm}\label{dual}(\cite{Caruso}) Let $\mathcal{G}$ be a commutative finite flat group scheme satisfying the assumption of theorem (\ref{Breuil}) and $\mathcal{G}^{\star}$ be the dual group scheme. There is a canonical isomorphism
$\Mod(\mathcal{G}^{\star})\simeq(\Mod(\mathcal{G}))^{\star}$.
\end{thm}
Using this result and the analogue of lemma (\ref{main}) for torsion modules, we obtain the following corollary.

\begin{cor}\label{iso2} Let $\widetilde{M},\;\widetilde{M}'$ be two strongly divisible modules of weight one. Let $n\geq 1$ be a positive integer. We consider the torsion modules $\displaystyle M=\frac{\widetilde{M}}{p^{n}}$ and $\displaystyle M'=\frac{\widetilde{M}'}{p^{n}}$. Then we have a canonical isomorphism
\[\Hom_{\Mod_{[0,2]}^{\phi}}(\mathbf{1}^{n}_{1}[-1],M\otimes M')\simeq\Hom_{ffgps/\mathcal{O}_{K}}(\Gp(M'),\Gp(M)^{\star}).\]
\end{cor}
\begin{proof}  The first step is to imitate lemma (\ref{main}) and prove an isomorphism
\begin{eqnarray}F:\Hom_{\Mod_{[0,2]}^{\phi}}(\mathbf{1}^{n}_{1}[-1],M\otimes M')\stackrel{\simeq}{\longrightarrow}\Hom_{\Mod_{[0,1]}^{\phi}}(M^{\star},M').
\end{eqnarray}
If we have this, then the result follows by the anti-equivalence of categories as follows.
\begin{eqnarray*}\Hom_{\Mod_{[0,2]}^{\phi}}(\mathbf{1}^{n}_{1}[-1],M\otimes M')\simeq&&\Hom_{\Mod_{[0,1]}^{\phi}}(M^{\star},M')\\
\simeq&&\Hom_{ffgps/\mathcal{O}_{K}}(\Gp(M'),\Gp(M^{\star}))\\ \simeq&&\Hom_{ffgps/\mathcal{O}_{K}}(\Gp(M'),\Gp(M)^{\star}).
\end{eqnarray*} To prove the isomorphism \ref{iso2}, we can reproduce the proof of lemma \ref{main} in the torsion setting. For example the hard step was to show that for a section $e\in\Fil^{1}(M\otimes M')$ with $\phi_{1}(e)=e$, the morphism $F_{e}:M^{\star}\rightarrow M'$ commutes with the Frobenius. Since $\Fil^{1}M=\Fil^{1}\widetilde{M}/p^{n}\Fil^{1}\widetilde{M}$, we may assume there is an element $\tilde{e}\in\Fil^{1}\widetilde{M}$ whose reduction modulo $p^{n}$ coincides with $e$. Then the proof for the strongly divisible module case carries over, with only difference that if we have two expressions for $e$ as follows, $e=\sum_{i,j}s_{ij}e_{i}\otimes w_{j}$ and $e=\sum_{i,j}t_{ij}x_{i}\otimes y_{j}$ for some $s_{ij},t_{ij}\in S$, then $t_{ij}-s_{ij}$ is equal to the table \ref{table} modulo $p^{n}$.

\end{proof}
\vspace{1pt}
\section{Integral Crystalline cohomology} In this section we want to pass from syntomic cohomology to crystalline. The main reference for this section is \cite{Fal}. For a smooth scheme $\mathfrak{X}$ over $\Spec (\mathcal{O}_{K})$, we consider the crystalline cohomology groups of $\mathfrak{X}$ relative to the ring $S=W[[u]]^{PD}$, which is the ring defined in section 5. We will denote these groups by $H^{i}_{Br}(\mathfrak{X}/S)$. The definition is similar to \cite{Berthelot1974}, Ch.III, 1 and uses infinitesimal thickenings $\mathcal{U}$ of open subsets $U\subset X$ such that $\mathcal{U}$ is a scheme over $S$ and the ideal defining $U$ in $\mathcal{U}$ has a nilpotent PD-structure compatible with that on $\Fil^1 S$ and $(p)$. For more details we refer to \cite{Fal}, section 2, page 2. 

The groups $H^{i}_{Br}(\mathfrak{X}/S)$  can be computed in the derived category using de Rham complexes as follows.
Suppose a global embedding $\mathfrak{X}\hookrightarrow \mathcal{Z}$ exists, where $\mathcal{Z}$ is smooth over $\Spec(S)$ (this is true when $\mathfrak{X}$ is an abelian scheme, which is the case of interest to us). We consider the de Rham complex
\[\mathbb{O}_{\mathfrak{X}\subset \mathcal{Z}}:=\mathcal{O}_{\mathfrak{X}\subset \mathcal{Z}}\otimes_{\mathcal{O}_{\mathcal{Z}}}\Omega_{\mathcal{Z}/S}^{\bullet},\] where $\mathcal{O}_{\mathfrak{X}\subset \mathcal{Z}}$ is the completed PD-envelope of $\mathfrak{X}$ in $\mathcal{Z}$. Then we have an isomorphism
\[H^{i}_{Br}(\mathfrak{X}/S)\simeq H^{i}(\mathfrak{X}_{et},\mathbb{O}_{\mathfrak{X}\subset \mathcal{Z}}).\]
For $r\leq p-1$ we consider the complex
$\mathbb{J}_{\mathfrak{X}\subset \mathcal{Z}}^{[r]}:=\mathcal{J}^{[r-\bullet]}_{\mathfrak{X}\subset \mathcal{Z}}\otimes_{\mathcal{O}_{\mathcal{Z}}}\Omega^{\bullet}_{\mathcal{Z}/S},$ where $\mathcal{J}^{[1]}_{\mathfrak{X}\subset \mathcal{Z}}$ is the PD-ideal generated by the kernel of the surjection $\mathcal{O}_{\mathcal{Z}}\twoheadrightarrow\mathcal{O}_{\mathfrak{X}}$. Next, the Frobenius $\phi$ on $S$ induces a Frobenius, $\phi^S$ on $\mathcal{Z}$. We therefore get a Frobenius, which we will denote again by $\phi^S$, on the PD-envelope $\mathcal{O}_{X\subset \mathcal{Z}}$ defined similarly to $\phi:S\rightarrow S$, by sending the constants $c_i$ to $\phi^S(c_i)$ and sending $f\in\ker[\mathcal{O}_{\mathcal{Z}}\twoheadrightarrow\mathcal{O}_{\mathfrak{X}}]$ to $f^p$. 

If we reduce the complex $\mathbb{O}_{\mathfrak{X}\subset \mathcal{Z}}$ modulo the augmentation ideal $\Fil^{1}S$ of the ring $S$ (this means we consider the derived tensor product with $S/\Fil^{1}S$), we obtain the Hodge cohomology groups, $H^{i}(\mathfrak{X}_{et},\mathcal{O}_{\mathfrak{X}/S}^{cris}
\otimes\Omega^{\bullet}_{\mathfrak{X}/\mathcal{O}_{K}})$, of $\mathfrak{X}$.
We make the following assumption on $\mathfrak{X}$.
\begin{ass}\label{assumption} The Hodge cohomology groups of $\mathfrak{X}$ are torsion free over $\mathcal{O}_{K}$ and the Hodge spectral sequence degenerates.
\end{ass}
If $\mathfrak{X}$ is smooth proper over $\mathcal{O}_{K}$ and the assumption (\ref{assumption}) is satisfied, then the following properties are known for the crystalline cohomology groups (\cite{Fal}, theorem 1).
\begin{enumerate}\item For $0\leq i\leq r\leq p-2$, the quadruple $(H^{i}_{Br}(\mathfrak{X}/S), H^{i}(\mathfrak{X}_{et},\mathbb{J}_{\mathfrak{X}\subset Z_{1}}^{[r]},\phi_{r},N)$ defines a strongly divisible module over $S$ of weight $r$ (see rem. \ref{conne} for the definition of the derivation $N$).
\item If we reduce modulo $p^{n}$, for $n\geq 1$, we obtain an isomorphism
\[H^{i}_{Br}(\mathfrak{X}/S)/p^{n}\simeq H^{i}(\mathfrak{X}_{n},\mathbb{O}_{\mathfrak{X}\subset Z_{1}}\otimes^{\mathbb{L}}\Z/p^{n}):=
H^{i}_{Br}(\mathfrak{X}_{n}/S_{n}).\]
\end{enumerate}
\begin{conv} Let $i<p-1$. From now on we identify the crystalline cohomology group $H^{i}_{Br}(\mathfrak{X}/S)$ with the Breuil module $(H^{i}_{Br}(\mathfrak{X}/S),\Fil^{i}H^{i}_{Br}(\mathfrak{X}/S),\phi_{i},N)\in
\Mod^{\phi, N}_{[0,i]}$.
\end{conv}
\begin{rem}\label{conne} We briefly describe how the derivation $N$ is obtained. For simplicity we assume that a global embedding $\mathfrak{X}\hookrightarrow\mathcal{Z}$ exists. There is a short exact sequence of complexes, \[0\rightarrow\Omega^{\bullet-1}_{\mathcal{Z}/S}\otimes\Omega_{S/W}^{1}\rightarrow\Omega^{\bullet}_{\mathcal{Z}/W}
\rightarrow\Omega^{\bullet}_{\mathcal{Z}/S}\rightarrow 0.\] It is well known that $\displaystyle\Omega^{1}_{S/W}\simeq S\frac{du}{u}$.   The above sequence induces a connecting homomorphism
\[H^{i}_{Br}(\mathfrak{X}/S)\rightarrow H^{i+1}(\mathfrak{X},\mathcal{O}_{\mathfrak{X}\subset\mathcal{Z}}\Omega^{\bullet-1}_{\mathcal{Z}/S}\otimes_{S}\Omega_{S/W}^{1})
\simeq H^{i}_{Br}(\mathfrak{X}/S)\otimes_{S} \frac{du}{u}.\] This map gives a connection $\nabla$ and hence a derivation $N$ with the required properties.
\end{rem}

\subsection{The crystalline cohomology of an abelian scheme} From now on we focus on the case when $\mathcal{A}$ is an abelian scheme over $\Spec(\mathcal{O}_{K})$. We note that the assumption (\ref{assumption}) is satisfied (prop. 2.5.2 in \cite{Berthelot/Breen/Messing1982}). But in this case we can in fact say more about the structure of the groups $H^{i}_{Br}(\mathcal{A}/S)$.

\subsection*{Computations for $H^{1}$} We first consider the weight one torsion module $H^{1}_{Br}(\mathcal{A}_{n}/S_{n})$. The following proposition highlights the relation between crystalline cohomology and finite flat group schemes associated to $p$-divisible groups.
\begin{prop}\label{H1} For every $n\geq 1$, there is a canonical isomorphism of Breuil modules \[H^{1}_{Br}(\mathcal{A}_{n}/S_{n})\simeq\Mod(\mathcal{A}[p^{n}]).\]
\end{prop}
\begin{proof} Let $H=\lim\limits_{\longrightarrow}\mathcal{A}[p^{n}]$ be the $p$-divisible group over $\mathcal{O}_{K}$ associated to the abelian scheme $\mathcal{A}$. By \cite{Bre1}, theorem 4.2.2.9, there is an associated strongly divisible module, defined as,
\[M(H)=(\lim\limits_{\longleftarrow}\Mod(\mathcal{A}[p^{n}]),\lim\limits_{\longleftarrow}
\Fil^{1}\Mod(\mathcal{A}[p^{n}]),\lim\limits_{\longleftarrow}\phi_{1}).\]
On the other hand, Faltings in section 6 of \cite{Fal} associates another strongly divisible module $M'(H)$ to the $p$-divisible group $H$. We claim that there is a canonical isomorphism $M(H)\simeq M'(H)$. This follows by the isomorphism, $T_{st}^{\star}(M(H))\simeq T_{st}^{\star}(M'(H))$, of the corresponding  Galois representations. In fact, both are isomorphic to the Tate module $T_{p}(H)$ of $H$. The isomorphism $T_{st}^{\star}(M(H))\simeq T_{p}(H)$ is proved in lemma 5.3.1 of \cite{Bre1}, while $T_{st}^{\star}(M'(H))\simeq T_{p}(H)$ is proved in theorem 7 of \cite{Fal}. Notice that the functor $T_{st}^{\star}$ becomes fully faithful when restricted to the filtered free modules (\cite{Fal} theorem 5).

Next we consider the strongly divisible module $H^{1}_{Br}(\mathcal{A}/S)$. Faltings showed (\cite{Fal}, page 132) that the Galois representations $T_{st}^{\star}(H)$ and $T_{st}^{\star}(H^{1}_{Br}(\mathcal{A}/S))$ coincide. We therefore obtain an isomorphism of strongly divisible modules
\[H^{1}_{Br}(\mathcal{A}/S)\simeq (\lim\limits_{\longleftarrow}\Mod(\mathcal{A}[p^{n}]),\lim\limits_{\longleftarrow}
\Fil^{1}\Mod(\mathcal{A}[p^{n}]),\lim\limits_{\longleftarrow}\phi_{1}).\]
The next step is to go modulo $p^{n}$. For the module $M(H)$ of the right hand side, it follows from proposition 4.2.1.2 of \cite{Bre1}, that $M(H)/p^{n}\simeq\Mod(\mathcal{A}[p^{n}])$.
Finally, for the module on the left we have an isomorphism $H^{1}_{Br}(\mathcal{A}/S)/p^{n}\simeq H^{1}_{Br}(\mathcal{A}_{n}/S_{n})$, as mentioned in the beginning of this section.

Notice that in the above discussion we did not include the connection $\nabla$. The reason is that for weight one modules we have a full faithful embedding $\Mod_{1}^{\phi,N}\subset\Mod_{1}^{\phi}$. In particular, the monodromy operator $N$ of the weight one module $M=H^{1}_{Br}
(\mathcal{A}_{n}/S_{n})$ has the characterizing property $\displaystyle N(M)\subset\sum_{i\geq 1}\frac{u^{i}}{q(i)!}M$, which makes it unique (\cite{Bre1}, prop. 5.1.3). According to the remark in \cite{Bre2} that follows lemma 3.2.1, every map of filtered modules whose monodromy operator $N$ satisfies the above property commutes with the monodromy, as long as it commutes with the Frobenius and therefore the above isomorphism is an isomorphism in $\Mod_1^{\phi,N}$.

\end{proof}
\subsection*{Computations for $H^{2}$} Next we consider the Breuil module $H^{2}_{Br}(\mathcal{A}/S)$. We want to express it in terms of $H^{1}$. We start with the following proposition, which follows by the deformation theory of abelian varieties (section 7 in \cite{Fal}). 
\begin{prop} There exists a smooth abelian scheme $\mathcal{Z}$ over $S$ and an embedding $\mathcal{A}\hookrightarrow\mathcal{Z}$ such that we have  a fiber product square
 \[ \xymatrix{\mathcal{A}\ar[r]\ar[d] & \mathcal{Z}\ar[d]\\
\Spec(\mathcal{O}_{K})\ar[r] & \Spec(S).
}
\]
\end{prop}
\begin{proof}  We will only sketch the steps that follow from sections 6 and 7 of   (\cite{Fal}). We consider  the $p$-divisible group $H$ arising from the abelian scheme $\mathcal{A}$. Moreover, similarly to the proposition \ref{H1} we consider the Breuil module $M(H)$ (filtered Frobenius-crystal in the language of \cite{Fal}) corresponding to $H$. Note that $M(H)$ without filtration depends only on the reduction  of $H$ modulo $p$. 

We consider the special fiber $\overline{H}=H\otimes_{\mathcal{O}_K}k\simeq H\otimes_{W(k)}k$. Now we follow the steps at the end of page 134 of \cite{Fal}. Namely, $H$ is a $p$-divisible group over $S/\Fil^1 S$ that deforms $\overline{H}$ and therefore we can lift it to a $p$-divisible group $\tilde{H}$ over $S$. The rest follows from \cite{Fal} (th. 10).  

\end{proof}
 Since the definition of $H^{i}_{Br}(\mathcal{A}/S)$ does not depend on the embedding, the above fiber square yields an isomorphism of Breuil modules
\[H^{i}_{Br}(\mathcal{A}/S)\simeq H^{i}_{dR}(\mathcal{Z}/S)=H^{i}(\mathcal{Z}_{et},\Omega^{\bullet}_{\mathcal{Z}/S}).\] Moreover, the Hodge cohomology of $\mathcal{A}$ coincides with the de Rham cohomology $H^{i}_{dR}(\mathcal{A}/\mathcal{O}_{K})$.
We have a natural $S$-linear map $\rho:\bigwedge^{2}H^{1}_{dR}(\mathcal{Z}/S)\rightarrow H^{2}_{dR}(\mathcal{Z}/S)$. We claim that this is a morphism of Breuil modules. Notice that the filtration on $\bigwedge^{2}H^{1}_{dR}(\mathcal{Z}/S)$ is defined as follows.
\begin{itemize}\item $\displaystyle\Fil^{1}(\bigwedge^{2}H^{1}_{dR}(\mathcal{Z}/S))$ is the image in $\bigwedge^{2}H^{1}_{dR}(\mathcal{Z}/S)$ of the module \[\Fil^{1}(H^{1}_{dR}(\mathcal{Z}/S))\otimes H^{1}_{dR}(\mathcal{Z}/S)+H^{1}_{dR}(\mathcal{Z}/S)\otimes \Fil^{1}(H^{1}_{dR}(\mathcal{Z}/S)),\] where $\displaystyle\Fil^{1}(H^{1}_{dR}(\mathcal{Z}/S)=H^{1}(\mathcal{Z}_{et},(\Fil^{1}S)\cdot
\mathcal{O}_{\mathcal{Z}}\stackrel{d}{\longrightarrow}\Omega^{1}_{\mathcal{Z}/S}
\stackrel{d}{\longrightarrow}
\Omega^{2}_{\mathcal{Z}/S}\longrightarrow\cdots)
$.
\item $\displaystyle\Fil^{2}(\bigwedge^{2}H^{1}_{dR}(\mathcal{Z}/S))=
\bigwedge^{2}\Fil^{1}(H^{1}_{dR}(\mathcal{Z}/S)).$
 \end{itemize} On the other hand, the filtration on $H^{2}_{dR}(\mathcal{Z}/S)$ is defined as follows.
\begin{itemize}\item $\Fil^{1}(H^{2}_{dR}(\mathcal{Z}/S))=H^{2}(\mathcal{Z}_{et},\mathbb{J}^{[1]}_{\mathcal{Z}/S})$, where $\mathbb{J}^{[1]}_{\mathcal{Z}/S}$ is the complex \[(\Fil^{1}S)\cdot
\mathcal{O}_{\mathcal{Z}}\stackrel{d}{\longrightarrow}\Omega^{1}_{\mathcal{Z}/S}
\stackrel{d}{\longrightarrow}
\Omega^{2}_{\mathcal{Z}/S}\longrightarrow\cdots
.\]
\item
$\Fil^{2}(H^{2}_{dR}(\mathcal{Z}/S))=H^{2}(\mathcal{Z}_{et},\mathbb{J}^{[2]}_{\mathcal{Z}/S})$, where $\mathbb{J}^{[1]}_{\mathcal{Z}/S}$ is the complex
\[(\Fil^{2}S)\cdot\mathcal{O}_{\mathcal{Z}}\stackrel{d}
{\longrightarrow}
(\Fil^{1}S)\cdot\mathcal{O}_{\mathcal{Z}}\otimes_{\mathcal{O}_{\mathcal{Z}}}
\Omega^{1}_{\mathcal{Z}/S}\stackrel{d}
{\longrightarrow}\Omega^{2}_{\mathcal{Z}/S}\longrightarrow\cdots.\]
\end{itemize} It is clear that $\rho(\Fil^{i}(\bigwedge^{2}H^{1}_{dR}(\mathcal{Z}/S)))\subset\Fil^{i}H^{2}_{dR}(\mathcal{Z}/S)$, for $i=1,2$. Moreover, $\rho$ commutes with the Frobenius map, since $\phi_{2}$ on $\Omega^{2}_{\mathcal{Z}/S}$ is defined as $\phi_{1}\wedge\phi_{1}$. It is also clear that $\rho$ commutes with the connection $\nabla$. We conclude that $\rho$ is a map of Breuil modules.
\begin{thm}\label{H2} The map $\rho:\bigwedge^{2}H^{1}_{Br}(\mathcal{A}/S)\longrightarrow H^{2}_{Br}(\mathcal{A}/S)$ is an isomorphism in $\Mod_{[0,2]}^{\phi,N}$.
\end{thm}
\begin{proof} It suffices to show the following two claims.
\begin{enumerate}[(i)]\item $\rho$ is an isomorphism on the underlying $S$-modules.
\item $\rho$ induces isomorphisms on the associated graded pieces. Namely   $gr^{0}(\bigwedge^{2}H^{1}_{dR}(\mathcal{Z}/S)))\simeq gr^{0}H^{2}_{dR}(\mathcal{Z}/S)$ and
$gr^{1}(\bigwedge^{2}H^{1}_{dR}(\mathcal{Z}/S)))\simeq gr^{1}H^{2}_{dR}(\mathcal{Z}/S)$.
\end{enumerate}

\underline{Step 1:} We show that $\rho:\bigwedge^{2}H^{1}_{dR}(\mathcal{Z}/S)))\rightarrow H^{2}_{dR}(\mathcal{Z}/S)$ is an isomorphism of  $S$-modules. \\
This follows directly from \cite{Berthelot/Breen/Messing1982}, corollary 2.5.5. 

\underline{Step 2:} We show that $\rho$ induces an isomorphism on $gr^{0}$. \\
We start by computing $gr^{0}(H^{2}_{dR}(\mathcal{Z}/S))=H^{2}_{dR}(\mathcal{Z}/S)/\Fil^{1}H^{2}_{dR}(\mathcal{Z}/S)$. We consider the short exact sequence of complexes
$0\rightarrow\mathbb{J}^{[1]}_{\mathcal{Z}/S}\rightarrow\Omega^{\star}_{\mathcal{Z}/S}\rightarrow\mathbb{H}\rightarrow 0,$ where $\mathbb{H}$ is the complex
$\mathcal{O}_{\mathcal{Z}}/\Fil^{1}(S)\rightarrow 0\rightarrow 0$. We therefore obtain a long exact sequence
\[\cdots\rightarrow H^{2}(\mathcal{Z}_{et},\mathbb{J}^{[1]}_{\mathcal{Z}/S}))\rightarrow H^{2}(\mathcal{Z}_{et},\Omega^{
\star}_{\mathcal{Z}/S}))\rightarrow H^{2}(\mathcal{Z}_{et},\mathbb{H})\rightarrow\cdots\] Notice that $\mathcal{O}_{\mathcal{Z}}/\Fil^{1}(S)\simeq\mathcal{O}_{\mathcal{A}}$. We already know by the Breuil module structure on $H^{i}_{dR}(\mathcal{Z}/S)$ that the maps $H^{i}(\mathcal{Z}_{et},\mathbb{J}^{[1]}_{\mathcal{Z}/S}))\rightarrow H^{i}(\mathcal{Z}_{et},\Omega^{
\star}_{\mathcal{Z}/S}))$ are injective for $i=2,3$, hence we conclude that $gr^{0}(H^{2}_{dR}(\mathcal{Z}/S))\simeq H^{2}(\mathcal{O}_{\mathcal{A}}).$
On the other hand, we can very easily see that $gr^{0}(\bigwedge^{2}H^{1}_{dR}(\mathcal{Z}/S))\simeq\bigwedge^{2}gr^{0}H^{1}_{dR}(\mathcal{Z}/S)$ and by a very similar computation of complexes as above, we conclude that
$gr^{0}(\bigwedge^{2}H^{1}_{dR}(\mathcal{Z}/S))\simeq\bigwedge^{2}H^{1}(\mathcal{O}_{\mathcal{A}}).$

Note that it suffices to prove that the natural map
$\bigwedge^{2}H^{1}(\mathcal{O}_{\mathcal{A}})\rightarrow H^{2}(\mathcal{O}_{\mathcal{A}})$ is injective. For, it is automatically surjective, since it is induced from the filtration preserving surjective map $\rho$. To prove injectivity, we proceed as in step 1. It suffices to prove injectivity after  $\otimes\Q_{p}$, and then the result follows by the structure of tensor algebra on $H^{\star}(\mathcal{O}_{\mathcal{A}})$ (since $\mathcal{O}_{\mathcal{A}}$ is a Hopf algebra, the same argument as in step 1 applies).

\underline{Step 3:} We show that $\rho$ is an isomorphism on $gr^{1}$.
This is the most involved part of the proof. First, we compute $gr^{1}(H^{2}_{dR}(\mathcal{Z}/S))=\Fil^{1}H^{2}_{dR}(\mathcal{Z}/S)/\Fil^{2}H^{2}_{dR}(\mathcal{Z}/S)$. We consider the short exact sequence of complexes
$0\rightarrow\mathbb{J}^{[2]}_{\mathcal{Z}/S}\rightarrow\mathbb{J}^{[1]}_{\mathcal{Z}/S}\rightarrow\mathbb{H}'\rightarrow 0,$ where $\mathbb{H}'$ is the complex,
$(\Fil^{1}S/\Fil^{2}S)\cdot\mathcal{O}_{\mathcal{Z}}\stackrel{d}{\longrightarrow} \mathcal{O}_{\mathcal{A}}\otimes_{\mathcal{O}_{\mathcal{Z}}}\Omega^{1}_{\mathcal{Z}/S}\rightarrow 0$. We need to understand the differential $d$. A section of $(\Fil^{1}S/\Fil^{2}S)\cdot\mathcal{O}_{\mathcal{Z}}$ is of the form $E(u)f$, for some function $f\in\mathcal{O}_{\mathcal{Z}}$. Then $d(E(u)f)=fE'(u)du+E(u)df=fE'(u)du\in\mathcal{O}_{\mathcal{A}}\otimes_{\mathcal{O}_{\mathcal{Z}}}\Omega^{1}_{\mathcal{Z}/S}$. Thus, the map $d$ is just multiplication by $E'(u)$ (notice that $\Fil^{1}S/\Fil^{2}S\simeq S/\Fil^{1}S$, where the inverse map is given by multiplication by $E(u)$). Imitating the proof of step 2, we conclude that
\[gr^{1}(H^{2}_{dR}(\mathcal{Z}/S))\simeq H^{2}(\mathcal{O}_{\mathcal{A}}\stackrel{\cdot E'(u)}{\longrightarrow} \mathcal{O}_{\mathcal{A}}\otimes_{\mathcal{O}_{\mathcal{Z}}}\Omega^{1}_{\mathcal{Z}/S}\rightarrow 0).\]
We move on to compute the cohomology of this double complex. Let $0\rightarrow I^{\bullet}$ and $0\rightarrow J^{\bullet}$ be injective resolutions of $\mathcal{O}_{\mathcal{A}}$ and $\mathcal{O}_{\mathcal{A}}\otimes_{\mathcal{O}_{\mathcal{Z}}}\Omega^{1}_{\mathcal{Z}/S}$ respectively. The map $d$ induces chain maps $\Gamma(\mathcal{Z},I^{n})\rightarrow\Gamma(\mathcal{Z},J^{n})$, which are all given by multiplication by $E'(u)$. The complex $\mathbb{H}'$ on degree $n$ is given by
\[\Gamma(\mathcal{Z},I^{n})\oplus\Gamma(\mathcal{Z},J^{n-1})\stackrel{\partial}{\longrightarrow}
\Gamma(\mathcal{Z},I^{n+1})\oplus\Gamma(\mathcal{Z},J^{n}),\] where
$\partial(x,y)=(dx, xE'(u)du-dy)$. Therefore, $H^{2}(\mathbb{H}')$ consists of classes of pairs $[x,y]\in\Gamma(\mathcal{Z},I^{2})\oplus\Gamma(\mathcal{Z},J^{1})$ such that the following two properties hold.
\begin{itemize}
\item $dx=0$
\item $dy=xE'(u)du$
\end{itemize}
and the equivalence relation is described by the image of $\Gamma(\mathcal{Z},I^{1})\oplus\Gamma(\mathcal{Z},J^{0})$.

Next we consider the associated graded $gr^{1}(\bigwedge^{2}H^{1}_{dR}(\mathcal{Z}/S))$. As in step 2, it suffices to show that $\rho$ maps this group injectively into $gr^{1}(H^{2}_{dR}(\mathcal{Z}/S))$. We want to describe a set of basis elements of this group. We start by choosing an adapted basis $\{e_{i}\}_{1\leq i\leq d}$ of $H^{1}_{dR}(\mathcal{Z}/S)$. We assume that $e_{i}\in\Fil^{1}H^{1}_{dR}(\mathcal{Z}/S)$ for $i\leq d_{1}$ and $e_{i}\not\in\Fil^{1}H^{1}_{dR}(\mathcal{Z}/S)$, for $i>d_{1}$. Then $gr^{1}(\bigwedge^{2}H^{2}_{dR}(\mathcal{Z}/S))$ is freely generated by the image of the following two families, $\mathcal{B}_{1}=\{e_{i}\wedge e_{j}:i\leq d_{1},\;j>d_{1}\}$ and $\mathcal{B}_{2}=\{E(u)e_{k}\wedge e_{j}:j>k>d_{1}\}$.

The first observation is that the image of the elements $e_{k}\wedge e_{j}$ for $j>k>d_{1}$ in $\bigwedge^{2}gr^{0}H^{1}_{dR}(\mathcal{Z}/S)$ form an $\mathcal{O}_{K}$-basis of $\bigwedge^{2}gr^{0}H^{1}_{dR}(\mathcal{Z}/S)$ which by step 2 is isomorphic to $H^{2}(\mathcal{O}_{\mathcal{A}})$. We conclude that $\mathcal{B}_{2}$ is an $\mathcal{O}_{K}$-basis of $E(u)\cdot H^{2}(\mathcal{O}_{\mathcal{A}})$.

Next, we observe that we have an injection $E(u)\cdot H^{2}(\mathcal{O}_{\mathcal{A}})\hookrightarrow H^{2}(\mathbb{H}')$ given by $x\rightarrow[x,0]$.

Similarly we have an injection $H^{1}(\mathcal{O}_{\mathcal{A}}\otimes_{\mathcal{O}_{\mathcal{Z}}}
\Omega^{1}_{\mathcal{Z}/S})\hookrightarrow H^{2}(\mathbb{H}')$ given by $y\rightarrow[0,y]$. With a similar computation as above we can show that $\mathcal{B}_{1}$ forms an $\mathcal{O}_{K}$-basis of $H^{1}(\mathcal{O}_{\mathcal{A}}\otimes_{\mathcal{O}_{\mathcal{Z}}}
\Omega^{1}_{\mathcal{Z}/S})$, which concludes the proof.

\end{proof}
\subsection{From syntomic to crystalline} The last ingredient we need in order to finish the proof of theorem \ref{BIG2} is a map
$H^2(\mathcal{A}_{syn},\mu_{p^n})\rightarrow\Hom_{\Mod_{[0,2]}^{\phi}}(\mathbf{1}_{n}[-1],H^2_{Br}(\mathcal{A}_n/S_n))$.

We consider again the embedding $\mathcal{A}\hookrightarrow\mathcal{Z}$, where $\mathcal{Z}$ is an abelian scheme over $\Spec(S)$. Notice that we have a map of complexes $\mathbb{J}^{[1]}_{\mathcal{A}\subset\mathcal{Z}}\stackrel{1-\phi_1^S}{\longrightarrow}\mathbb{O}_{\mathcal{A}\subset\mathcal{Z}}$. We consider the kernel,

\[S_n(1)^{Br}=\ker(1-\phi_1^S)=\Cone(\mathbb{J}^{[1]}_{\mathcal{A}/\mathcal{Z}}\otimes^{\mathbb{L}}\Z/p^n
\stackrel{1-\phi_1^S}{\longrightarrow}
\mathbb{O}_{\mathcal{A}/\mathcal{Z}}\otimes^{\mathbb{L}}\Z/p^n)[-1].\] We can now restate the goal of this subsection. We want to show the existence of a map
$H^2(\mathcal{A}_{syn},\mu_{p^n})\rightarrow H^2(\mathcal{A}_{et},S_n(1)^{Br})$. We need a bridge through the crystalline-syntomic site, which is defined as follows. For every $n\geq 1$, we have embeddings $\mathcal{A}_1\hookrightarrow\mathcal{Z}_n$, where $\mathcal{A}_{1}=\mathcal{A}\times\Z/p$ and $\mathcal{Z}_n=\mathcal{Z}\times\Z/p^n$. The small crystalline-syntomic site $(\mathcal{A}_{1}/S_n)_{crys-syn}$ has objects  quadruples  $(U,T,i,\delta)$, where $U\rightarrow\mathcal{A}_1$ is a syntomic map, $U\stackrel{i}{\hookrightarrow}T$ is a closed immersion that fits into a commutative diagram,
$\xymatrix{
U\ar[r]^{i}\ar[d] & T\ar[d]\\
\mathcal{A}_1\ar[r] &\mathcal{Z}_n,
}$, and $\delta$ is  a structure of an ideal of divided powers corresponding to the closed immersion $i$ and compatible with the PD-structure on $\mathcal{A}_1\hookrightarrow\mathcal{Z}_n$. For simplicity we will write the objects as pairs $(U,T)$. There are two natural structure sheaves on this site, namely $\mathcal{O}_{n, cris}$, given by $\mathcal{O}_{n, cris}(U,T):=\mathcal{O}(T)$ and $\mathcal{O}_{\mathcal{A}_n}$, given by $\mathcal{O}_{\mathcal{A}_n}(U,T):=\mathcal{O}(U)$. We have a surjection
$\mathcal{O}_{n, crys}\twoheadrightarrow\mathcal{O}_{\mathcal{A}_n}$ with kernel $\mathcal{J}_n$, a sheaf of PD-ideals. Similarly to the above situation, we have the syntomic-crystalline sheaf,
$S_n(1)^{crys-syn}:=\ker(\mathcal{J}_n\stackrel{1-\phi_1}{\longrightarrow}\mathcal{O}_{n, cris}).$ As in the syntomic situation considered in section 4, it is known that there is an isomorphism
$H^2(\mathcal{A}_{et},S_n(1)^{Br})\simeq H^2(\mathcal{A}_{crys-syn},S_n(1)^{crys-syn}).$

We next imitate the argument from Fontaine-Messing (\cite{Fon-Mess}, page 199), namely we construct a map of crystalline-syntomic sheaves  $\mu_{p^n}\rightarrow S_n(1)^{crys-syn}$ using the logarithm.  We briefly sketch the argument here. 

Let $(U,T)\in(\mathcal{A}_{1}/S_n)_{crys-syn}$ and consider a section $\zeta\in\mu_{p^n}(U)\subset\mathcal{O}(U)=\mathcal{O}_{\mathcal{A}_n}(U)$. Using the surjection $\mathcal{O}(T)\twoheadrightarrow \mathcal{O}(U)$, we lift $\zeta$ to a section $\zeta'\in\mathcal{O}(T)=\mathcal{O}_{n,crys}(U,T)$. We clearly have $\zeta'^{p^n}\in 1+\mathcal{J}_n(U,T)$.  Moreover, its logarithm, $\log(\zeta'^{p^n})$ is well defined and lies in $\mathcal{J}_n$. We define in this way a map $\mu_{p^n}(U,T)\rightarrow \mathcal{J}_n(U,T)$. Next notice that $\log(\zeta'^{p^n})$ is an eigenvector of the Frobenius $\phi^S$, namely $p\log(\zeta'^{p^n})=\log((\zeta'^{p^n})^p)=\phi^S(\log(\zeta'^{p^n})$. Therefore, $\log(\zeta'^{p^n})\in\ker(\mathcal{J}_n\stackrel{p-\phi^S}{\longrightarrow}\mathcal{O}_{n,cris})$. By passing to the derived categories, we can conclude that $\log(\zeta'^{p^n})$ lies in $S_n(1)^{crys-syn}(U,T)$ yielding a map of crystalline-syntomic sheaves, $\mu_{p^n}\rightarrow S_n(1)^{crys-syn}$.


To finish the argument, we consider the natural inclusion of topologies $\mathcal{A}_{syn}\stackrel{\nu}{\rightarrow}\mathcal{A}_{crys-syn}$, that sends a syntomic covering $\{U_i\rightarrow\mathcal{A}\}$ to $(U_i,T_i)$, where $T_i$ is the trivial PD-thickening. This induces a map on cohomology $H^i(\mathcal{A}_{syn},\mu_{p^n})\rightarrow H^i(\mathcal{A}_{crys-syn},\mu_{p^n})$. Here we identified the syntomic sheaf $\mu_{p^n}$ with the pullback $\nu^\star(\mu_{p^n})$ of the crystalline-syntomic $\mu_{p^n}$.

\begin{rem} We note that if we wanted to avoid the derived categories language, we should have proceeded slightly differently, taking more care of the divided Frobenius, $\phi_1^S$. Namely, imitating Fontaine-Messing, define the crystalline-syntomic sheaf $S_n(1)^{crys-syn}$ as the image of the natural map $\tilde{S}_{n+1}^{crys-syn}\rightarrow\tilde{S}_{n}^{crys-syn}$, where $\tilde{S}_{n}^{crys-syn}=\ker(\mathcal{J}_n\stackrel{p-\phi^S}{\longrightarrow}\mathcal{O}_{n,cris})$. Then we should instead start with a section $\zeta\in\mu_{p^{n+1}}(U)$ and a lift $\zeta'\in\mathcal{O}(T)$ and at the end use some diagram chasing to obtain the desired map. 
\end{rem}
\begin{rem}\label{forget} We can in fact show that the map just obtained factors through the group
$\Hom_{\Mod_{[0,2]}^{\phi,N}}(\mathbf{1}_{n}[-1],H^2_{Br}(\mathcal{A}_n/S_n))$. Although we won't need this fact, we note that the map, \[\Hom_{\Mod_{[0,2]}^{\phi,N}}(\mathbf{1}_{1}[-1]^{n},H^{2}_{Br}
(\mathcal{A}_{n}/S_{n}))\rightarrow \Hom_{\Mod_{[0,2]}^{\phi}}(\mathbf{1}_{1}[-1]^{n},H^{2}_{Br}
(\mathcal{A}_{n}/S_{n})),\] obtained by the forgetful functor
is in fact an isomorphism. For, by theorem \ref{H2} we have an isomorphism of torsion Breuil modules $H^{2}_{Br}
(\mathcal{A}_{n}/S_{n})\simeq \bigwedge^{2} H^{1}_{Br}
(\mathcal{A}_{n}/S_{n})$. The unique property of the monodromy $N$ on $H^{1}_{Br}(\mathcal{A}_{n}/S_{n})$ mentioned in the proof of proposition \ref{H1} passes to the tensor product $M\otimes M$ and hence to $H^{2}_{Br}
(\mathcal{A}/S)$. The monodromy on the weight one module $\mathbf{1}_{1}^{n}[-1]$ clearly has the same property and therefore the forgetful map is an isomorphism.
\end{rem}
\vspace{1pt}
\section{Proof of the main Theorem} In this section we complete the proof of theorem \ref{BIG2}. We first need the following definition.

\subsection*{Symmetric Homomorphisms} Let $M$ be a filtered free module. We consider the module $\bigwedge^{2}M$. Since 2 is invertible in $S$, we can think of $\bigwedge^{2}M$ as a direct summand of $M\otimes M$. For, the map $\displaystyle\alpha\wedge\beta\rightarrow\frac{\alpha\otimes\beta-\beta\otimes\alpha}{2}$ gives a splitting $\bigwedge^{2}M\hookrightarrow M\otimes M$. We have a similar splitting for the torsion module $M\otimes M=\tilde{M}/p^{n}\otimes\tilde{M}/p^{n}$, where $\tilde{M}$ is a strongly divisible module of weight one. We give the following definition.

\begin{defn}\label{sym} Let $\mathcal{G}$ be a finite flat group scheme over $\mathcal{O}_{K}$ such that $p^{n}$ annihilates $\mathcal{G}$. A homomorphism $f:\mathcal{G}\rightarrow\mathcal{G}^{\star}$ of finite flat groups schemes over $\mathcal{O}_{K}$ is called symmetric if the corresponding morphism of Breuil modules, $\mathbf{1}_{1}^{n}[-1]\rightarrow (\Mod(\mathcal{G}))^{\otimes 2}$,  obtained by lemma \ref{iso2}, factors through $\bigwedge^{2}\Mod(\mathcal{G})$. Equivalently, $f:\mathcal{G}\rightarrow\mathcal{G}^{\star}$ is symmetric if and only if $f^\star=-f$, where $f^\star$ is the dual homomorphism. We will denote the subgroup of such homomorphisms by $Sym\mathcal{H}om_{ffgps/\mathcal{O}_{K}}(\mathcal{G},\mathcal{G}^{\star})$.
\end{defn}
\begin{rem} We note that if $\mathcal{A}$ is an abelian scheme and $f:\mathcal{A}[p^n]\rightarrow\mathcal{A}^\star[p^n]$ is a symmetric homomorphism, then the corresponding map of the generic fibers $f:A[p^n]\rightarrow A^\star[p^n]$ induces a $G_K$-homomorphism $\bigwedge^2 A[p^n]\rightarrow\mu_{p^n}$. For, the isomorphism $A^\star[p^n]\simeq\Hom(A[p^n],\mu_{p^n})$ is obtained via the Weil pairing, which is antisymmetric. If for example $J$ is the Jacobian variety of a smooth complete curve over $K$ (and hence self dual), the identity map $J[p^n]\rightarrow J[p^n]$ is symmetric in the above sense. This is the reason we chose the name symmetric.
\end{rem}
The following corollary is an immediate consequence of definition \ref{sym}.
\begin{cor}\label{final} Let $\mathcal{G}$ be a finite flat group scheme over $\mathcal{O}_{K}$ such that $p^{n}$ annihilates $\mathcal{G}$. There is a canonical isomorphism
\[Sym\mathcal{H}om_{ffgps/\mathcal{O}_{K}}(\mathcal{G},\mathcal{G}^{\star})\simeq\Hom_{\Mod^{\phi}_{[0,2]}}
(\mathbf{1}_{1}^{n}[-1],\bigwedge^{2}\Mod(\mathcal{G})).\]
\end{cor}

\subsection{Completion of the proof}
\begin{proof} (of theorem \ref{BIG2}).\\
The first step is to refine corollary \ref{cyclic} and express the orthogonal complement of the image of the cycle map in terms of homomorphisms of finite flat group schemes.\\
\underline{Step 1:}
We show that the orthogonal complement under the Tate duality pairing $(\star\star)$ of the image of the cycle map $c_{p^{n}}:T(A)\longrightarrow H^{2}(K,\bigwedge^{2}A[p^{n}])$ is the image of the injection
\[Sym\mathcal{H}om_{ffgps/\mathcal{O}_{K}}(\mathcal{A}[p^{n}],\mathcal{A}^{\star}[p^{n}])\hookrightarrow
\Hom_{G_{K}}(\bigwedge^{2}A[p^{n}],\mu_{p^{n}}).\] Notice that because of proposition \ref{annih}, it suffices to show that the orthogonal complement in question is a subset of $Sym\mathcal{H}om_{ffgps/\mathcal{O}_{K}}(\mathcal{A}[p^{n}],\mathcal{A}^{\star}[p^{n}])$.

By corollary \ref{cyclic} we know that this orthogonal complement is the image of the composition
\[H^{2}_{syn}(\mathcal{A},\mu_{p^{n}})\rightarrow H^{2}(A,\mu_{p^{n}})\rightarrow \Hom_{G_{K}}(\bigwedge^{2}A[p^{n}],\mu_{p^{n}}). \]
It therefore suffices to show that this map factors through $Sym\mathcal{H}om_{ffgps/\mathcal{O}_{K}}(\mathcal{A}[p^{n}],\mathcal{A}^{\star}[p^{n}])$.
We have a map $H^{2}_{syn}(\mathcal{A},\mu_{p^{n}})\rightarrow\Hom_{\Mod_{[0,2]}^{\phi,N}}(\mathbf{1}_{1}[-1]^{n},H^{2}_{Br}
(\mathcal{A}_{n}/S_{n}))$ (see subsection 6.2 and remark \ref{forget}).
Next, by theorem \ref{H2} and corollary \ref{final} we have an isomorphism,
\[\Hom_{\Mod_{[0,2]}^{\phi}}(\mathbf{1}_{1}[-1]^{n},H^{2}_{Br}
(\mathcal{A}_{n}/S_{n}))\simeq Sym\mathcal{H}om_{ffgps/\mathcal{O}_{K}}(\mathcal{A}[p^{n}],\mathcal{A}^{\star}[p^{n}]).\] By the way we constructed the maps, the following diagram is commutative
\[\xymatrix{
& H^{2}_{syn}(\mathcal{A},\mu_{p^{n}})\ar[r]\ar[d]& Sym\mathcal{H}om_{ffgps/\mathcal{O}_{K}}(\mathcal{A}[p^{n}],\mathcal{A}^{\star}[p^{n}])\ar[d]\\
& H^{2}(A,\mu_{p^{n}})\ar[r] & \Hom_{G_{K}}(\bigwedge^{2}A[p^{n}],\mu_{p^{n}}),
\\
}\] which completes the first step.

\underline{Step 2:} We pass to the symbolic part $H^{2}_{s}(K,A[p^{n}]\otimes B[p^{n}])$ by considering the product $A\times B$.
For the abelian variety $X=A\times B$, step 1 gives us the orthogonal complement under the pairing $(\star\star)$ of $\img[S_{2}(K;A\times B)\stackrel{s_{p^{n}}}{\longrightarrow}H^{2}(K,\bigwedge^{2}(A\times B)[p^{n}])]$.

For abelian varieties $A_{1},\cdots,A_{r}$ over $K$, the Somekawa $K$-group $K(K;A_{1},\cdots,A_{r})$ is known to satisfy covariant functoriality (\cite{Som}, remark before theorem 1.4). This functoriality  yields a decomposition
\[K(K;A\times B, A\times B)\simeq K(K;A,A)\oplus K(K;A,B)\oplus K(K;B,B).\] If we divide with the action of the symmetric group in two variables, we get a decomposition \[S_2(K;A\times B)\simeq S_2(K;A)\oplus K(K;A,B)\oplus S_2(K;B),\] where we recall that $S_2(K;A)$ is the quotient of $K(K;A,A)$ by the action of the symmetric group in two variables. 
On the other hand, we have an isomorphism of $G_{K}$-modules
\[\bigwedge^{2}(A\times B)[p^{n}]\simeq\bigwedge^{2}A[p^{n}]\oplus A[p^{n}]\otimes B[p^{n}]\oplus \bigwedge^{2}B[p^{n}].\]
We have a similar decomposition of the group $Sym\mathcal{H}om_{ffgps/\mathcal{O}_{K}}((\mathcal{A}\times\mathcal{B})[p^{n}],(\mathcal{A}
\times\mathcal{B})^{\star}[p^{n}])$.
The theorem then follows, after we observe that under the map $s_{p^{n}}$, the direct summands $S(K;A,A)$ and $S(K;B,B)$ map to $H^{2}(K,\bigwedge^{2}A[p^{n}])$ and $H^{2}(K,\bigwedge^{2}B[p^{n}])$ respectively, for which we know the orthogonal complements by step 1, leaving the piece $K(K;A,B)\stackrel{s_{p^{n}}}{\longrightarrow}H^{2}(K,A[p^{n}]\otimes B[p^{n}])$, whose orthogonal complement has to be $\mathcal{H}om_{ffgps/\mathcal{O}_{K}}(\mathcal{A}[p^{n}],\mathcal{B}^{\star}[p^{n}])$.

\end{proof}
\vspace{1pt}
\section{Applications}
\subsection{Triviality of the Galois symbol when the ramification index is small} In the last subsection we will provide specific examples where the Galois symbol is non trivial. On the other extreme, we have the following vanishing result, which is a direct consequence of theorem \ref{BIG2} combined with a theorem of Raynaud.
\begin{cor} Suppose that $p>3$. Let $e$ be the ramification index of $K$ and assume that $e<p-1$. For abelian varieties $A,B$ over $K$ with good reduction, the Galois symbol $K(K;A,B)\stackrel{s_{p^n}}{\longrightarrow}H^2(K,A[p^n]\otimes B[p^n])$ vanishes.
\end{cor}
\begin{proof} Raynaud (\cite{Raynaud}) proved that when $e<p-1$, the faithful functor $\mathcal{G}\rightarrow\mathcal{G}\times_{\mathcal{O}_K}K$ that sends a finite flat group scheme over $\mathcal{O}_K$ to its generic fiber is also full. Therefore, we have an isomorphism
$\mathcal{H}om_{ffgps/\mathcal{O}_{K}}(\mathcal{A}[p^{n}],\mathcal{B}^{\star}[p^{n}])\simeq\Hom_{G_K}(A[p^n],B^{\star}[p^n])$ and the corollary follows directly from theorem \ref{BIG2}.

\end{proof}
\subsection{Passing to the limit} In this subsection we want to discuss the limit behavior of the Galois symbol. Let us first assume that $A,B$ are any abelian varieties over $K$, not necessarily of good reduction. First, for positive integers $n,m\geq 1$ it is easy to show that we have a commutative diagram as follows,
\[ \xymatrix{
s_{mn}:K(K;A,B)\ar[r]\ar[rd]_{s_{m}} & H^2(K,A[mn]\otimes B[mn])\ar[d]^{n}\\
& H^2(K,A[m]\otimes B[m]).
}\] Therefore, we obtain a symbol map $K(K;A,B)\rightarrow\lim\limits_{\longleftarrow}H^2(K,A[m]\otimes B[m])$, where the limit extends through all positive integers $m$.

On the other hand, the Tate duality pairing $(\star)$ after passing to the limit becomes,
\[(\star')\;\lim\limits_{\longleftarrow}H^2(K,A[m]\otimes B[m])\times\lim\limits_{\longrightarrow}\Hom_{G_K}(A[m],B^{\star}[m])\rightarrow\Q/\Z.\]
Note that the pairing $(\star')$ is not necessarily perfect anymore. In particular, if the group  $\lim\limits_{\longleftarrow}H^2(K,A[m]\otimes B[m])$ has torsion, then its torsion  lies in the left kernel of the pairing.
\begin{lem} Let $T(A), T(B^\star)$ be the total Tate modules of $A,B^\star$ respectively. We have a canonical isomorphism $\lim\limits_{\longrightarrow}\Hom_{G_K}(A[m],B^{\star}[m])\simeq\Hom_{G_K}(T(A),T(B)).$
\end{lem}
\begin{proof} First we can write $\Hom_{G_K}(A[m],B^{\star}[m])\simeq
\Hom_{G_K}(A[m]\otimes B[m],\Q/\Z(1))$. Next we show that $\lim\limits_{\longrightarrow}$ commutes with $G_K$-invariants. This follows easily by the fact that the morphisms $A[m]\hookrightarrow A[mn]$ are injections and $\lim\limits_{\longrightarrow}$ is exact. The lemma will thus follow, if we show an isomorphism $\lim\limits_{\longrightarrow}\Hom(A[m]\otimes B[m],\Q/\Z(1))
\simeq\Hom(\lim\limits_{\longleftarrow}(A[m]\otimes B[m]),\Q/\Z(1))$. But the latter is true, since $\lim\limits_{\longleftarrow}(A[m]\otimes B[m])\simeq T(A)\otimes T(B)\simeq\widehat{\Z}^{\oplus 2d_1}\otimes\widehat{\Z}^{\oplus 2d_2}$ and we know that $\widehat{\Z}^\vee\simeq\Q/\Z$.

\end{proof}
\begin{cor} Let $A,B$ be abelian varieties over $K$ with good reduction. The image of the Galois symbol $K(K;A,B)\rightarrow\lim\limits_{\longleftarrow}H^2(K,A[m]\otimes B[m])$ lies on the left kernel of the pairing
\[(\star')\;\lim\limits_{\longleftarrow}H^2(K,A[m]\otimes B[m])\times\Hom_{G_K}(T(A),T(B^{\star}))\rightarrow\Q/\Z.\] Moreover, it is annihilated by a power of $p$.
\end{cor}
\begin{proof} Let $\mathcal{H}_{A}=\lim\limits_{\longrightarrow}\mathcal{A}[p^n]$ be the $p$-divisible group corresponding to the N\'{e}ron model $\mathcal{A}$ of $A$, and similarly $\mathcal{H}_{B^\star}$. By the theorem of Tate (\cite{Tate1}) we have an isomorphism \[\Hom_{G_K}(T_p(A),T_p(B^{\star})\simeq\mathcal{H}om_{pdiv}(\mathcal{H}_A,\mathcal{H}_B^\star).\] Recall that for $m$ coprime to $p$ the Galois symbol $s_m$ is trivial. The first claim of the corollary then follows by this isomorphism and by theorem \ref{BIG2}.

For the second claim, we use a computation of Bondarko (\cite{Bond}). The latter shows that the functor $\mathcal{G}\rightarrow\mathcal{G}\times_{\mathcal{O}_K}K$ is "weakly full" in the following sense. If $f:A[p^n]\rightarrow B^\star[p^n]$ is a homomorphism of $G_K$-modules, there exists a nonnegative integer $s$ that depends only on the absolute ramification index of $K$ and a homomorphism $g:\mathcal{A}[p^n]\rightarrow\mathcal{B}^\star[p^n]$ of finite flat group schemes over $\mathcal{O}_K$ such that $p^s f$ coincides with $g$ after restriction to the generic fiber. This in particular implies, that the image of the Galois symbol is $p^s$-torsion and hence finite.

\end{proof} 
Using the relation with zero cycles described in section 3, we immediately get the following corollary. 
\begin{cor}\label{finite2} Let $K$ be a finite extension of $\Q_p$ with $p>3$. Let $A$ be an abelian variety of dimension $d$ over $K$ with good reduction. The  cycle map  to \'{e}tale cohomology,
$CH_{0}(A)\rightarrow H^{2d}_{et}(A,\Z_p(1)^{\otimes d})$, when restricted to the Albanese kernel $T(A)$, has finite image.
\end{cor}
\vspace{1pt}
\subsection{Examples of non trivial Galois symbol}
We finish this draft by describing examples of abelian varieties $A,B$ of specific reduction type such that the symbolic part, $H^2_s(K,A[p^n]\otimes B[p^n])$, is non trivial. We do not exhaust all possible cases, but rather we focus on generalizing cases that have been previously considered by other authors (\cite{Mur/Ram}, \cite{Hir3}). Our method is to describe more explicitly the inclusion \[\mathcal{H}om_{ffgps/\mathcal{O}_{K}}(\mathcal{A}[p^{n}],\mathcal{B}^{\star}[p^{n}])\subset
\Hom_{G_K}(A[p^n],B^{\star}[p^n]),\] using the decomposition of a finite flat group scheme into connected and \'{e}tale components. This method can apply to many more cases. The main philosophy is that when at least one of the finite flat group schemes $\mathcal{A}[p^{n}],\mathcal{B}^{\star}[p^{n}]$ has non trivial \'{e}tale quotient, this inclusion is very often strict when the base field $K$ is large enough.

In this section we will write for simplicity $\mathcal{H}om_{\mathcal{O}_{K}}(\mathcal{A}[p^{n}],\mathcal{B}^{\star}[p^{n}])$ to denote the group of homomorphisms of finite flat group schemes over $\Spec(\mathcal{O}_K)$.
\subsection*{Connected and \'{e}tale components} Let $\mathcal{G}$ be a finite flat group scheme over $\mathcal{O}_K$. It is well known (see for example \cite{Tate1}) that there is a short exact sequence of finite flat group schemes over $\mathcal{O}_K$,
\[0\rightarrow\mathcal{G}^\circ\rightarrow\mathcal{G}\rightarrow\mathcal{G}^{et}\rightarrow 0,\] where $\mathcal{G}^\circ$ is the connected component of the zero element of $\mathcal{G}$ and $\mathcal{G}^{et}$ is the \'{e}tale quotient.
\begin{notn} If $G=\mathcal{G}\times_{\mathcal{O}_K}K$ is the generic fiber of $\mathcal{G}$, we will use the notation $G^{\circ}:=\mathcal{G}^{\circ}\times_{\mathcal{O}_K}K$ and $G^{et}:=\mathcal{G}^{et}\times_{\mathcal{O}_K}K$.
\end{notn}
\subsection{The ordinary reduction Case}  In this section we consider abelian varieties $A,B$ over $K$ with good ordinary reduction. We first review the definition of this type of reduction.
\begin{defn} Let $A$ be an abelian variety over $K$ of dimension $d$. We say that $A$ has good ordinary reduction if it has good reduction and the special fiber
$\overline{A}:=\mathcal{A}\times_{\mathcal{O}_K}k$ is an ordinary abelian variety over the residue field $k$. This means that the group $\overline{A}[p^n](\overline{k})$ is the largest possible, namely there is an isomorphism $\overline{A}[p^n](\overline{k})\simeq(\Z/p^n)^{\oplus d}$.
\end{defn}
From now on we assume that the base field $K$ contains a primitive $p$-th root of unity. (Note that the Galois symbol is trivial otherwise).

The following lemma describes the structure of the finite flat group scheme $\mathcal{A}[p^n]$, at least in the case when $\mathcal{A}[p^n]$ is self dual (which is true for example when $A$ is the Jacobian variety of a smooth complete curve over $K$).
\begin{lem}\label{ordin} Assume $A$ is of dimension $d$ and has good ordinary reduction. Then the connected and \'{e}tale components of $\mathcal{A}[p^n]$ are both non trivial.  Assume additionally that the finite flat group scheme $\mathcal{A}[p^n]$ is self dual. In this case, there is a  finite unramified extension $L/K$ such that $\mathcal{A}[p^n]^{\circ}_{\mathcal{O}_L}\simeq(\mu_{p^n}^{\oplus d})_{\mathcal{O}_L}$ and $\mathcal{A}[p^n]^{et}_{\mathcal{O}_L}\simeq(\Z/p^n)_{\mathcal{O}_L}^{\oplus d}$.
\end{lem}
\begin{proof} The lemma follows by reduction to the residue field. Since $\overline{A}[p^n](\overline{k})\simeq
(\Z/p^n)^{\oplus d}$, the group scheme $\mathcal{A}[p^n]$ cannot be either connected or \'{e}tale. Moreover, we can consider a finite unramified extension $L$ over $K$ such that all the points of $\overline{A}[p^n]$ become $l$-rational, where $l$ is the residue field of $L$. This yields an isomorphism of finite flat group schemes over $\mathcal{O}_L$, $\mathcal{A}[p^n]^{et}_{\mathcal{O}_L}\simeq(\Z/p^n)_{\mathcal{O}_L}^{\oplus d}$. Since we assumed that $\mathcal{A}[p^n]$ is self dual, the claim for the connected component follows by Cartier duality.

\end{proof}

\begin{prop}\label{ordin2} Let $A,B$ be abelian varieties over $K$ with good ordinary reduction of dimensions $d_{1},d_{2}$ respectively. Assume that both $\mathcal{A}[p^n]$ and $\mathcal{B}[p^n]$ are self dual. The image of the inclusion $\mathcal{H}om_{\mathcal{O}_{K}}(\mathcal{A}[p^{n}],\mathcal{B}^{\star}[p^{n}])\subset\Hom_{G_K}(A[p^n],B^{\star}[p^n])$ is precisely the subgroup \[H:=\{f\in\Hom_{G_K}(A[p^n],B^{\star}[p^n]):f(A[p^n]^{\circ})\subset B^{\star}[p^n]^{\circ}\}.\]
\end{prop}
\begin{proof} The inclusion $\mathcal{H}om_{\mathcal{O}_{K}}(\mathcal{A}[p^{n}],\mathcal{B}^{\star}[p^{n}])\subset H$ follows from the more general fact that  there are no non trivial homomorphisms from a connected to an \'{e}tale group scheme over $\mathcal{O}_K$. It suffices to show the other inclusion after a base change to a finite extension $L$ of $K$. For, by finite flat descent we obtain an isomorphism
\[\mathcal{H}om_{\mathcal{O}_K}(\mathcal{A}[p^{n}] ,\mathcal{B}^{\star}[p^{n}])\simeq(\mathcal{H}om_{\mathcal{O}_L}(\mathcal{A}[p^{n}]\times_{\mathcal{O}_K} \mathcal{O}_{L},\mathcal{B}^{\star}[p^{n}]\times_{\mathcal{O}_K} \mathcal{O}_{L})^{\Gal(L/K)}.\]

We may therefore assume that we are in the set up of lemma \ref{ordin} and we have short exact sequences of finite flat group schemes over $\mathcal{O}_K$,
$0\rightarrow\mathcal{A}[p^n]^{\circ}\rightarrow\mathcal{A}[p^n]\rightarrow\mathcal{A}[p^n]^{et}\rightarrow 0$
with $\mathcal{A}[p^n]^{\circ}\simeq(\mu_{p^n})_{\mathcal{O}_K}^{\oplus d_{1}}$ and $\mathcal{A}[p^n]^{et}\simeq(\Z/p^n)_{\mathcal{O}_K}^{\oplus d_1}$ and similarly for $\mathcal{B}^\star[p^n]$.
Next, we can easily see by counting dimensions that we have a short exact sequence of $G_K$-modules
\[0\rightarrow\Hom(A[p^n]^{et},B^{\star}[p^{n}])\rightarrow\widetilde{H}
\rightarrow\Hom(A[p^n]^{\circ},B^{\star}[p^n]^{\circ})\rightarrow 0,\] where $\widetilde{H}=\{f\in\Hom(A[p^n],B^{\star}[p^n]):f(A[p^n]^{\circ})\subset B^{\star}[p^n]^{\circ}\}$.

\underline{Claim:} There is a short exact sequence of flat sheaves over $\mathcal{O}_{\overline{K}}$,
\begin{eqnarray*}&&0\rightarrow\mathcal{H}om_{\mathcal{O}_{\overline{K}}}(\mathcal{A}[p^n]^{et},\mathcal{B}^{\star}[p^{n}])\rightarrow
\mathcal{H}om_{\mathcal{O}_{\overline{K}}}(\mathcal{A}[p^{n}],\mathcal{B}^{\star}[p^{n}])\rightarrow
\mathcal{H}om_{\mathcal{O}_{\overline{K}}}(\mathcal{A}[p^n]^\circ,\mathcal{B}^{\star}[p^n]^{\circ})\rightarrow 0.
\end{eqnarray*} To prove the claim we apply the left exact functor $\mathcal{H}om_{\mathcal{O}_{\overline{K}}}(-,\mathcal{B}^{\star}[p^{n}])$ to the short exact sequence of finite flat sheaves over $\mathcal{O}_{\overline{K}}$,
$0\rightarrow\mathcal{A}^{\star}[p^{n}]^{\circ}\rightarrow \mathcal{A}^{\star}[p^{n}]\rightarrow\mathcal{A}^{\star}[p^{n}]^{et}\rightarrow 0$. Notice that the latter sequence splits. For, there is finite extension $F\supset K$ such that $A[p^n]\subset A(F)$. Hence, over $\mathcal{O}_F$ we have a splitting $A[p^n]\simeq A[p^n]^{\circ}\oplus A[p^n]^{et}$. Since $A[p^n]^{et}\simeq(\Z/p^n)^{\oplus d_{1}}$, the splitting $A[p^n]^{et}\hookrightarrow A[p^n]$ corresponds to $d_{1}$ points of $ A[p^n](F)$. By the N\'{e}ron model property we get corresponding sections of $\mathcal{A}[p^n](\mathcal{O}_F)$ and hence there is a splitting of finite flat group schemes over $\mathcal{O}_F$, $\mathcal{A}[p^n]\simeq\mathcal{A}[p^n]^{et}\oplus\mathcal{A}[p^n]^{\circ}$.

The two short exact sequences from above induce long exact sequences as follows,
\[ \xymatrix{
\mathcal{H}om_{\mathcal{O}_K}(\mathcal{A}[p^n]^{et},\mathcal{B}^{\star}[p^{n}])\ar[r]\ar[d]_{\beta} &
\mathcal{H}om_{\mathcal{O}_K}(\mathcal{A}[p^{n}] ,\mathcal{B}^{\star}[p^{n}])\ar[r]\ar[d]_{\gamma}& \mathcal{H}om_{\mathcal{O}_K}(\mathcal{A}[p^n]^{\circ},\mathcal{B}^{\star}[p^{n}]^{\circ})\ar[r]\ar[d]_{\varepsilon}&\dots\\
\Hom_{G_K}(A[p^n]^{et},B^{\star}[p^{n}])\ar[r] & H\ar[r] & \Hom_{G_K}(A[p^n]^{\circ},B[p^n]^{\circ})\ar[r] & \dots
}
\]
The next piece of the diagram is the map \[H^{1}_{fl}(\mathcal{O}_K,\mathcal{H}om(\mathcal{A}[p^n]^{et},\mathcal{B}^{\star}[p^{n}]))\stackrel{\eta}{\longrightarrow} H^{1}(K,\Hom(A[p^n]^{et},B^{\star}[p^{n}])).\] The maps $\beta,\gamma,\varepsilon$ are obtained by restriction to the generic fiber, so they are all inclusions. We claim that $\beta$ and $\varepsilon$ are in fact isomorphisms.
For, we have isomorphisms
\begin{eqnarray*}&&\Hom_{G_K}(A[p^n]^{et},B^{\star}[p^{n}])\simeq\bigoplus_{d_{1}}\Hom_{G_K}(\Z/p^n,B^{\star}[p^{n}])
\simeq\bigoplus_{d_{1}}B^{\star}[p^{n}](K)\\
&&\mathcal{H}om_{\mathcal{O}_K}(\mathcal{A}[p^n]^{et},\mathcal{B}^{\star}[p^{n}])\simeq\bigoplus_{d_{1}}
\mathcal{H}om_{\mathcal{O}_K}(\Z/p^n,\mathcal{B}^{\star}[p^{n}])\simeq\bigoplus_{d_{1}} \mathcal{B}^{\star}[p^{n}](\mathcal{O}_K).
\end{eqnarray*} The last two groups are isomorphic by the N\'{e}ron model property, and therefore $\beta$ is an isomorphism. The claim for $\varepsilon$ follows similarly by using Cartier duality. We have therefore reduced the problem to  showing an inclusion \[H^{1}_{fl}(\mathcal{O}_K,\mathcal{H}om(\mathcal{A}[p^n]^{et},\mathcal{B}^{\star}[p^{n}]))
\stackrel{\eta}{\hookrightarrow} H^{1}(K,\Hom(A[p^n]^{et},B^{\star}[p^{n}])),\] which can be rewritten as $\bigoplus_{d_1}H^{1}_{fl}(\mathcal{O}_K,\mathcal{B}^{\star}[p^{n}])\stackrel{\eta}{\hookrightarrow} \bigoplus_{d_1}H^{1}(K,B^{\star}[p^{n}]).$ It therefore suffices to prove an inclusion $H^{1}_{fl}(\mathcal{O}_K,\mathcal{B}^{\star}[p^{n}])\hookrightarrow H^{1}(K,B^{\star}[p^{n}])$. This follows by the vanishing of $H^1_{fl}(\mathcal{O}_K,\mathcal{B}^\star)$ (see for example the appendix of \cite{Mur/Ram}), which yields an isomorphism \[H^{1}_{fl}(\mathcal{O}_K,\mathcal{B}^{\star}[p^{n}])\simeq\mathcal{B}^{\star}(\mathcal{O}_K)/p^n
\simeq B^\star(K)/p^n\hookrightarrow H^1(K,B^\star[p^n]).\]

\end{proof}
The following corollary generalizes the results of Murre and Ramakrishnan (\cite{Mur/Ram}) and Hiranouchi and Hirayama (\cite{Hir3}) to higher dimensional abelian varieties.
\begin{cor}\label{ordin3} Let $A$, $B$ be abelian varieties over $K$ with good ordinary reduction. Assume that all the $p^n$ torsion points of $A$ and $B$ are $K$-rational. Then the symbolic subgroup $H^2_s(K,A[p^n]\otimes B[p^n])$  of $H^2(K,A[p^n]\otimes B[p^n])$ is isomorphic to $(\Z/p^n)^{d_1d_2}$.
\end{cor}
\begin{proof} By the semi-simplicity of the $G_K$ action on $A[p^n]$, $B[p^n]$, we get an isomorphism  \[\Hom_{G_K}(A[p^n],B^{\star}[p^n])\simeq(\Z/p^n)^{4d_1d_2},\] while the group $H$ of proposition \ref{ordin2} is isomorphic to $(\Z/p^n)^{3d_1d_2}$. The corollary then follows by theorem \ref{BIG2}.

\end{proof}
\begin{rem}\label{smallp} Proposition \ref{ordin2} combined with theorem \ref{BIG2} yield an isomorphism,  \[H^2_s(K,A[p^n]\otimes B[p^n])\simeq \img(H^2(K,A^{\circ}[p^n]\otimes B^{\circ}[p^n])\rightarrow H^2(K,A[p^n]\otimes B[p^n])).\] The inclusion $(\subset)$ follows easily by the fact that the image of the Galois symbol and the group $\mathcal{H}om_{\mathcal{O}_{K}}(\mathcal{A}[p^{n}],\mathcal{B}^{\star}[p^{n}])$ annihilate each other under the Tate duality pairing.

One can prove the inclusion $(\supset)$ directly, without the use of $p$-adic Hodge theory. This is closer to the approach used by previous authors. Namely, one can show that the image of the Galois symbol covers all of $H^2(K,A^{\circ}[p^n]\otimes B^{\circ}[p^n])$ by using the following fact. If $L/K$ is a finite extension such that $A[p^n]^{\circ}\simeq\mu_{p^n}^{\oplus d_1}$ and $B[p^n]^{\circ}\simeq\mu_{p^n}^{\oplus d_2}$, then
the cup product
\[H^1_{fl}(O_L, A[p^n]^{\circ})\otimes H^1_{fl}(O_L, B[p^n]^{\circ})\to H^2(L, \bigoplus_{d_1 d_2}\mu_{p^n}\otimes \mu_{p^n})\]
is identified with some copies of the Hilbert symbol map \[O_L^\times/(O_L^\times)^{p^n} \otimes O_L^\times/(O_L^\times)^{p^n}\to H^2(L, \mu_{p^n}\otimes \mu_{p^n})\]
which is known to be surjective by local class field theory. This computation is unconditional on the prime $p$, which makes us believe that theorem \ref{BIG2} should be true even for small primes.
\end{rem}
\vspace{1pt}
\subsection{Ordinary and Supersingular} We next consider the case when $A$ has ordinary reduction, while $B^\star$ has supersingular reduction.

We note that $B^\star$ having supersingular reduction means that its reduction $\overline{B}^\star$ to the residue field $k$ has the property $\overline{B}^\star[p^n](\overline{k})=0$, for every $n\geq 1$. In particular, the finite flat group scheme $\mathcal{B}^\star[p^n]$ is connected. (This follows by the classification of the simple finite flat group schemes over $k$). We next give the analogue of  proposition \ref{ordin2} for this case.
\begin{prop}\label{super} Let $A, B$ be  abelian varieties over $K$ of dimensions $d_1,d_2$  respectively. We assume that both $\mathcal{A}[p^n],\mathcal{B}[p^n]$ are self dual, $A$ has good ordinary reduction and $B$ has good supersingular reduction. If the base field $K$ is large enough so that we are in the set up of lemma \ref{ordin}, then we have canonical isomorphisms
\[\mathcal{H}om_{\mathcal{O}_{K}}(\mathcal{A}[p^{n}],\mathcal{B}^{\star}[p^{n}])\simeq
\mathcal{H}om_{\mathcal{O}_{K}}(\mathcal{A}[p^{n}]^{et},\mathcal{B}^{\star}[p^{n}])
\simeq\Hom_{G_K}(A[p^n]^{et},B^{\star}[p^n]).\]
\end{prop}
\begin{proof} First we show the equality  $\mathcal{H}om_{\mathcal{O}_{K}}(\mathcal{A}[p^{n}],\mathcal{B}^{\star}[p^{n}])=
\mathcal{H}om_{\mathcal{O}_{K}}(\mathcal{A}[p^{n}]^{et},\mathcal{B}^{\star}[p^{n}])
$. It suffices to show that if $f:\mathcal{A}[p^n]^\circ\rightarrow\mathcal{B}^\star[p^n]$ is a homomorphism over $\mathcal{O}_K$, then $f=0$. Indeed, if $f$ is non trivial, the dual homomorphism $f^\star:\mathcal{B}[p^n]\rightarrow\mathcal{A}[p^n]^{et}$ will also be non trivial. This is a contradiction, since $\mathcal{B}[p^n]$ is connected, while $\mathcal{A}[p^n]^{et}$ is \'{e}tale.

The second isomorphism was already shown as part of the proof of proposition \ref{ordin2}.

\end{proof}
We close this section with the following generalization of the computation in \cite{Hir3}.
\begin{cor} Assume we are in the set up of proposition \ref{super}. Moreover, assume that all the $p^n$-torsion points of $A,B$ are $K$-rational. Then the symbolic subgroup $H^2_s
(K,A[p^n]\otimes B[p^n])$ of $H^2(K,A[p^n]\otimes B[p^n])$ is isomorphic to $(\Z/p^n)^{2d_1d_2}$.
\end{cor}
\begin{proof} The proof is analogous to the proof of corollary \ref{ordin3}.
\end{proof}
\begin{rem} Similarly to the remark \ref{smallp}, one can prove theorem \ref{BIG2} directly in this case imitating the explicit Hilbert symbol computations of S. Bloch in \cite{Bloch2} (pg 247).
\end{rem}
\bigskip
\appendix
\section{Tensor-Hom Compatibility} 
The following lemma gives the expected compatibility between the tensor product of two flitered free modules of weight 1 and the corresponding $\Hom$. 
\begin{lem} Let $M,M'$ be filtered-free modules of weight 1. There is a canonical isomorphism
\[\Hom_{\Mod_{[0,2]}^{\phi}}(\mathbf{1}_{1}[-1],M\otimes M')\simeq \Hom_{\Mod_{[0,1]}^{\phi}}(M^{\star},M').\]
\end{lem}
\begin{proof} First notice that if we forget the derivation in lemma (\ref{twist}), we have an isomorphism
\[\Hom_{\Mod_{[0,2]}^{\phi}}(\mathbf{1}_{1}[-1],M\otimes M')\simeq\{e\in \Fil^{1}(M\otimes M'):\phi_{1}(e)=e\}.\] Since both $M$ and $M'$ are of weight 1, there exist adapted basis $\textbf{B}^{1}_{M}=\{e_{1},\cdots,e_{d}\}$ and $\textbf{B}^{1}_{M'}=\{w_{1},\cdots,w_{\nu}\}$ of $M,M'$ respectively
such that \[\Fil^{1}M=(\bigoplus_{i=1}^{d_{1}}S e_{i})\bigoplus(\bigoplus_{i=d_{1}+1}^{d}\Fil^{1}S e_{i}),\;\;\;\; Fil^{1}M'=(\bigoplus_{j=1}^{\nu_{1}}S w_{j})\bigoplus(\bigoplus_{j=\nu_{1}+1}^{\nu}\Fil^{1}S w_{j}).\] For $1\leq i\leq d$, we set $x_{i}=\phi_{1}(E(u)^{n_{i}}e_{i})$, where $n_{i}=0$, for $i\leq d_{1}$ and $n_{i}=1$ for $i>d_{1}$. Similarly we set $y_{j}=\phi_{1}(E(u)^{m_{j}}w_{j})$, where $m_{j}=0$, for $j\leq \nu_{1}$ and $m_{j}=1$, for $j>\nu_{1}$. By lemma (\ref{phi-basis}) we get that the set
$\textbf{B}^{2}_{M}=\{x_{i}\}_{1\leq i\leq d}$ (respectively
$\textbf{B}^{2}_{M'}=\{y_{j}\}_{1\leq j\leq \nu}$) forms another basis of $M$ (resp.  $M'$). We will first define a map
\[F:\{e\in \Fil^{1}(M\otimes M'):\phi_{1}(e)=e\}\longrightarrow\Hom_{\Mod_{[0,1]}^{\phi}}(M^{\star},M').\]
Let $e\in\Fil^{1}(M\otimes M')$. We can write $e$ uniquely in the form
$e=\sum_{i=1}^{d}\sum_{j=1}^{\nu} s_{ij}e_{i}\otimes w_{j},$
where $s_{ij}\in S$ and when $i>d_{1}$ and $j>\nu_{1}$ we moreover have $s_{ij}\in\Fil^{1}S$. For such an element $e$ we define the following homomorphism
\begin{eqnarray*}F_{e}:&&M^{\star}\longrightarrow M'\\&&
f\longrightarrow F_{e}(f)=\sum_{i,j}s_{ij}f(e_{i})w_{j}.
\end{eqnarray*}  The map $F_{e}$ is clearly $S$-linear and it is an easy linear computation that it does not depend on the choice of basis. Moreover, we can easily verify that $F_{e}$ preserves the filtration, namely $F_{e}(\Fil^{1}(M^{\star}))\subset\Fil^{1}(M)$. For, let $f\in\Fil^{1}(M^{\star})$. By definition this means that $f(\Fil^{1}M)\subset \Fil^{1}(S)$. We rewrite $F_{e}(f)=\sum_{j\leq \nu_{1}}\sum_{i}s_{ij}f(e_{i})w_{j}+\sum_{j>\nu_{1}}\sum_{i}s_{ij}f(e_{i})w_{j}$. By the choice of the base $\mathbf{B}^{1}_{M'}$, the first sum is clearly in $\Fil^{1}M'$. For the second sum we notice that if $i\leq d_{1}$, then $e_{i}\in\Fil^{1}M$ and hence $f(e_{i})\in\Fil^{1} S$, while if $i>d_{1}$ then $s_{ij}\in\Fil^{1}S$. So in every case we have the second sum also in $\Fil^{1}M'$.

\underline{Claim 1:} Assume in addition that $\phi_{1}(e)=e$. Then $F_{e}$ commutes with the Frobenius.\\ This claim is not so obvious, so we will prove it in great detail. We need to show that for every functional $f\in M^{\star}$, it holds $\phi_{1}(F_{e}(f))=F_{e}(\phi_{1}^{\star}(f))$. Using the basis $\mathbf{B}_{M}^{2}$ and $\mathbf{B}_{M'}^{2}$, we can write $e$ uniquely in the form $e=\sum_{i,j}t_{ij}x_{i}\otimes y_{j}$. We then compute,
\begin{eqnarray*}\phi_{1}(e)=&&\phi_{1}(\sum_{i,j}s_{ij}e_{i}\otimes w_{j})=\sum_{j,\;i\leq d_{1}}\phi_{1}(e_{i})\otimes\phi_{0}(s_{ij}w_{j})+\\
&&+\sum_{i>d_{1}, j\leq \nu_{1} }\phi_{0}(s_{ij}e_{i})\otimes\phi_{1}(w_{j})+\sum_{i>d_{1},j>\nu_{1}}\phi_{1}(s_{ij}e_{i})
\otimes\phi_{0}(w_{j})\\=&&\sum_{i\leq d_{1},j\leq \nu_{1}}\phi_{1}(e_{i})\otimes p\phi(s_{ij})\phi_{1}(w_{j})+\sum_{i\leq d_{1},j> \nu_{1}}\phi_{1}(e_{i})\otimes\phi(s_{ij})\frac{1}{c}\phi_{1}(E(u)w_{j})+\\
&&\sum_{i>d_{1},j\leq \nu_{1}}\frac{1}{c}\phi(s_{ij})\phi_{1}(E(u)e_{i})\otimes\phi_{1}(w_{j})+\\&&
+\sum_{i> d_{1},j> \nu_{1}}\phi_{1}(s_{ij})\frac{1}{c}\phi_{1}(E(u)e_{i})\otimes\frac{1}{c}\phi_{1}(E(u)w_{j})\\=&&
\sum_{i\leq d_{1},j\leq \nu_{1}}p\phi(s_{ij})x_{i}\otimes y_{j}+\sum_{i\leq d_{1},j> \nu_{1}}\frac{1}{c}\phi(s_{ij})x_{i}\otimes y_{j}+\\&&+\sum_{i> d_{1},j\leq \nu_{1}}\frac{1}{c}\phi(s_{ij})x_{i}\otimes y_{j}+\sum_{i> d_{1},j> \nu_{1}}\frac{1}{c^{2}}\phi_{1}(s_{ij})x_{i}\otimes y_{j}.
\end{eqnarray*}We conclude that the condition $\phi_{1}(e)=e$ yields the following equality for the coefficients
\begin{eqnarray}\label{table} t_{ij}=\left\{
           \begin{array}{ll}
             p\phi(s_{ij}), \;i\leq d_{1},j\leq\nu_{1} \\
             c^{-1}\phi(s_{ij}), \; i\leq d_{1},j>\nu_{1}\;or\;i>d_{1},j\leq\nu_{1}  \\
             c^{-2}\phi_{1}(s_{ij}), \;i> d_{1},j>\nu_{1} .
           \end{array}
         \right.
\end{eqnarray} Let $f\in M^{\star}$. To compute $F_{e}(\phi_{1}^{\star}(f))$ we use the basis $\{x_{i}\otimes y_{j}\}$. Namely, according to this basis we have $F_{e}(\phi_{1}^{\star}(f))=\sum_{i,j}t_{ij}\phi_{1}^{\star}(f)(x_{i})y_{j}$. Considering all different cases for $i,j$ as before and using the table for $t_{ij}$,  we can calculate
\begin{eqnarray*}F_{e}(\phi_{1}^{\star}(f))=&&
\sum_{i\leq d_{1},j\leq\nu_{1}}p\phi(s_{ij})\phi_{1}(f(e_{i}))\phi_{1}(w_{j})+\sum_{i\leq d_{1},j>\nu_{1}}\frac{1}{c}\phi(s_{ij})\phi_{1}(E(u)f(e_{i}))\phi_{1}(E(u)w_{j})+\\
&&\sum_{i> d_{1},j\leq\nu_{1}}\frac{1}{c}\phi(s_{ij})\phi_{1}(f(e_{i}))\phi_{1}(w_{j})
+\sum_{i>d_{1},j>\nu_{1}}\frac{1}{c^{2}}\phi_{1}(s_{ij})\phi_{1}(E(u)f(e_{i}))\phi_{1}(E(u)w_{j}).
\end{eqnarray*}
To compute $\phi_{1}(F_{e}(f))$ we use the basis $\{e_{i}\otimes w_{j}\}$. We have,
\begin{eqnarray*}\phi_{1}(F_{e}(f))=&&\phi_{1}(\sum_{i,j}s_{ij}f(e_{i})w_{j})=
\sum_{i,j\leq\nu_{1}}\phi(s_{ij}f(e_{i}))\phi_{1}(w_{j})+\sum_{i,j>\nu_{1}}
\phi_{1}(s_{ij}f(e_{i}))\phi_{0}(w_{j}).
\end{eqnarray*} If we break the two last sums into the usual cases, $i\leq d_{1}$, $i>d_{1}$ we can see that all the coefficients agree with the ones of $F_{e}(\phi_{1}^{\star}(f))$. We conclude that we can define a homomorphism
\begin{eqnarray*}F:&&\{e\in \Fil^{1}(M\otimes M'):\phi_{1}(e)=e\}\longrightarrow\Hom_{\Mod_{[0,1]}^{\phi}}(M^{\star},M')
\end{eqnarray*} by sending $e$ to $F_{e}$.

The next step is to define a map in the converse direction
\begin{eqnarray*}H:&&\Hom_{\Mod_{[0,1]}^{\phi}}(M^{\star},M')\longrightarrow\{e\in \Fil^{1}(M\otimes M'):\phi_{1}(e)=e\}.
\end{eqnarray*}
For $i\in\{1,\cdots,d\}$ we consider the dual basis $\{e_{i}^{\star}\}$, where $e_{i}^{\star}:M\rightarrow S$ sends $e_{i}$ to $1$ and $e_{j}$ to 0 for every $j\neq i$. We can easily see that this forms an adapted basis for $M^{\star}$ if we reverse the ordering of the basis elements. This means that the submodule $\Fil^{1}M^{\star}$ has a decomposition
\[\Fil^{1}M^{\star}=(\bigoplus_{i=1}^{d_{1}}\Fil^{1}S e_{i}^{\star})\bigoplus(\bigoplus_{i=d_{1}+1}^{d}S e_{i}^{\star}).\]
Let $f\in\Hom_{\Mod_{[0,1]}^{\phi}}(M^{\star},M')$.  For every $i=1,\cdots,d$, we can write   $f(e_{i}^{\star})=\sum_{j=1}^{\nu}s_{ij}w_{j}$, for a unique choice of  $s_{ij}\in S$. We can therefore define a map
\begin{eqnarray*}H:&&\Hom_{\Mod_{[0,1]}^{\phi}}(M^{\star},M')\longrightarrow M\otimes M'\\
&&[f:M\rightarrow M^{\star}]\longrightarrow e_{f}=\sum_{i,j}s_{ij}e_{i}\otimes w_{j}.
\end{eqnarray*} Once again we can easily check that $f$  does not depend on the choice of the basis. In particular, if we start with the basis $\mathbf{B}_{M}^{2}=\{x_{i}\}$ and  we express in a unique way $f(x_{i}^{\star})=\sum_{j}t_{ij}y_{j}$, then we get a second expression of $e_{f}=\sum_{ij}t_{ij}x_{i}\otimes y_{j}$.

First we claim that $e_{f}\in\Fil^{1}(M\otimes M')$. For, if either $i\leq d_{1}$ or $j\leq\nu_{1}$, then the summand $s_{ij}e_{i}\otimes w_{j}$ is clearly in $\in\Fil^{1}(M\otimes M')$. For the case $i>d_{1}$, notice that the function $f$ preserves the filtration and $e_{i}^{\star}\in\Fil^{1}M^{\star}$. This implies that $\sum_{j=1}^{\nu}s_{ij}w_{j}\in\Fil^{1}M'$, which in turn forces $s_{ij}\in \Fil^{1}S$, for $j>\nu_{1}$. We conclude that $H$ gives a map
\[\Hom_{\Mod_{[0,1]}^{\phi}}(M^{\star},M')\stackrel{H}{\longrightarrow} \Fil^{1}(M\otimes M').\]
\underline{Claim 2:} The element $e_{f}$ is fixed by the Frobenius $\phi_{1}$.\\
First we observe that for every $i\in\{1,\cdots,d\}$ we have an equality $x_{i}^{\star}=c^{-1}\phi_{1}^{\star}(E(u)^{1-n_{i}}e_{i}^{\star})$. Since the function $f$ commutes with the Frobenius, this in turn yields
\begin{eqnarray*}f(x_{i}^{\star})=&&c^{-1}f\phi_{1}^{\star}(E(u)^{1-n_{i}}e_{i}^{\star})=
c^{-1}\phi_{1}(f(E(u)^{1-n_{i}}e_{i}^{\star}))=c^{-1}\phi_{1}(E(u)^{1-n_{i}}f(e_{i}^{\star}))\\=
&&\sum_{j=1}^{\nu}c^{-1}\phi_{1}(E(u)^{1-n_{i}}s_{ij}w_{j})=\sum_{j=1}^{\nu}t_{ij}y_{j}.
\end{eqnarray*} By considering all the different cases for $i\leq d_{1}$, $i>d_{1}$, $j\leq\nu_{1}$ and $j>\nu_{1}$ and equating the coefficients, we conclude that
\[t_{ij}=\left\{
           \begin{array}{ll}
             p\phi(s_{ij}), \;i\leq d_{1},j\leq\nu_{1} \\
             c^{-1}\phi(s_{ij}), \; i\leq d_{1},j>\nu_{1}\;or\;i>d_{1},j\leq\nu_{1}  \\
             c^{-2}\phi_{1}(s_{ij}), \;i> d_{1},j>\nu_{1} ,
           \end{array}
         \right.
\] which is exactly the same table we found before. This means exactly that $\phi_{1}(e_{f})=e_{f}$, which completes the proof of the claim.

The last step is to show that $F$ and $H$ are inverse of each other. This is almost a tautology. To compute $H\circ F$, let $e=\sum_{i,j}s_{ij}e_{i}\otimes w_{j}$. Then to compute $H(F_{e})$ we first compute $F_{e}(e_{i}^{\star})=\sum_{j}s_{ij}w_{j}$, so we conclude that $H(F_{e})=\sum_{j}s_{ij}e_{i}\otimes w_{j}=e$, as expected. The equality $(F\circ H)(f)=f$ is similar, so we omit the proof.

\end{proof}

\bibliographystyle{amsalpha}

\bibliography{bibliography3}

\end{document}